\documentclass[reqno, 12pt]{amsart}

\usepackage[algoruled,lined,noresetcount,norelsize]{algorithm2e}

\usepackage{graphicx}
\usepackage{amsmath,amssymb,amsthm}
\usepackage{mathrsfs}
\usepackage{mathabx}\changenotsign
\usepackage{dsfont}
\usepackage{xcolor}
\usepackage[backref]{hyperref}
\hypersetup{
	colorlinks,
	linkcolor={red!60!black},
	citecolor={green!60!black},
	urlcolor={blue!60!black}
}
\usepackage[T1]{fontenc}
\usepackage{lmodern}
\usepackage[babel]{microtype}
\usepackage[english]{babel}

\usepackage{todonotes}

\usepackage{geometry}
\numberwithin{equation}{section}
\numberwithin{figure}{section}

\newtheorem{theorem}{Theorem}

\newtheorem{claim}[theorem]{Claim}

\usepackage{tikz}
\usepackage{xcolor}
\usepackage{framed}
\usepackage{mathtools}
\usetikzlibrary{positioning}
\usetikzlibrary{decorations.pathreplacing,calligraphy}

\newcommand{\C}{\mathcal{C}}
\newcommand{\B}{\mathcal{B}}
\newcommand{\F}{\mathcal{F}}

\renewcommand{\P}{\mathcal{P}}
\renewcommand{\epsilon}{\varepsilon}

\newtheorem{thm}{Theorem}[section]
\newtheorem{prop}{Proposition}[section]

\newtheorem{lem}{Lemma}[section]

\newenvironment{claimproof}[1]{\par\noindent\textit{Proof of the claim.} #1}{\hfill $\blacksquare$\vspace{3pt}}

\title{Creating triangles in Constructor-Blocker games}

\author{Chloé Boisson, Yannick Mogge, Aline Parreau, Théo Pierron}

\thanks{This research was supported by the ANR project P-GASE (ANR-21-CE48-0001-01).}

\begin{document}
	
	\maketitle

	\begin{abstract}
		Generalized Tur\'an problems investigate the maximization of the number of certain structures (typically edges) under some constraints in a graph. We study a game version of these problems, the Constructor-Blocker game. We mainly focus on the case where Constructor tries to maximize the number of triangles in her graph, while forbidding her to claim short paths or cycles. We also study a variant of this game, where we impose some planarity constraints on Constructor instead of forbidding certain subgraphs. For all games studied, we obtain (precise) asymptotics or upper and lower bounds.
	\end{abstract}

	\section{Introduction}
	
	\subsection{The Turán problem}
	
	One of the most studied problems in extremal graph theory is the so called Tur\'an problem. Given a family $\mathcal{F}$ of graphs, a graph is called \emph{$\F$-free}, if it does not contain any graph $F \in \F$ as a subgraph. We can now ask the following question: what is the maximum number of edges an $\F$-free graph on $n$ vertices can have?
	
	This number is called the \emph{extremal number} $ex(n,\mathcal{F})$, or simply $ex(n,F)$ if $\F$ consists of a single graph $F$. In \cite{mantel1907vraagstuk}, Mantel determined this quantity for a triangle $F=K_3$, and Tur\'an generalized this result in \cite{turan1941egy} to all complete graphs. Since this first result, the extremal number 
	has been widely studied for variety of graphs and families of graphs; see for instance \cite{Simonovits1997} for a survey.
	
	\medskip
	
	A next step in understanding $\mathcal{F}$-free graphs can be the study of their subgraphs. Given a graph $H$, denote by $ex(n,H,\F)$ the maximal number of copies of $H$ an $n$-vertex $\F$-free graph can have. The classical Tur\'an problem corresponds to the case where $H = K_2$. Although some mentions of this problem in particular cases had appeared earlier (for instance in \cite{Erdös1969}), the study of $ex(n,H,\F)$ really intensified in the last decade, starting with Alon and Shikhelman in \cite{alon_many_2016}, and since then many beautiful results have been proven (see e.g.~\cite{luo_maximum_2018}).

	\subsection{Constructor-Blocker games}
	
	In this paper, we study a game version of the generalized Tur\'an problem, which was first introduced by Patk\'os, Stojakovi\'c, and Vizer~\cite{patkos_constructor-blocker_2022} under the name \emph{Constructor-Blocker game}. The rules of this game are as follows. Given two fixed graphs $H$ and $F$, two players, called \emph{Constructor} and \emph{Blocker}, alternately claim an unclaimed edge of the complete graph $K_n$. Constructor can only claim an edge in such a way that her graph remains $F$-free, while Blocker can claim any unclaimed edge without further restrictions. The game ends when Constructor cannot claim any more edges or when all edges have been claimed. The \emph{score} of the game is the number of copies of $H$ in Constructor's graph at the end of the game. Constructor aims to maximize the score while Blocker aims to minimize it. We denote the score of the game by $g(n,H,F)$, assuming optimal play by both players and Constructor starting the game. For a family $\F$ of graphs, we similarly define $g(n,H,\F)$.
	
	
	\medskip
	
	To get some intuition of the relation between $ex(h,H,\F)$ and, $g(n,H,\F)$, observe that the trivial bound $g(n,H,\F) \leq ex(n,H,\F)$ always holds. Moreover, in the case that $H$ is an edge, then Constructor can secretly pick an $\F$-free graph that maximizes the number of edges, and during the game she will be able to claim at least half of the edges of this graph. Therefore, we have $\frac12 ex(n,K_2,\F) \leq g(n,K_2,\F) \leq ex(n,K_2,\F)$, which tells us that $ex(n,K_2,\F)$ and $g(n,K_2,\F)$ are of the same order of magnitude. Actually, the asymptotics of $g(n,K_2,\F)$ only depend on the chromatic number of $\F$ and were determined in the context of Avoider-Enforcer games in~\cite{Balogh_Avoider}.
	
	So far, the quantity $g(n,H,F)$ has mainly been studied in two papers (\cite{balogh_constructor-blocker_2024,patkos_constructor-blocker_2022}).
	In \cite{patkos_constructor-blocker_2022}, Patk\'os, Stojakovi\'c, and Vizer deal with the cases where $F$ is a star and $H$ is a smaller star or a tree, and where both $H$ and $F$ are paths. 
	In \cite{balogh_constructor-blocker_2024}, Balogh, Chen, and English extended these results to the cases where $H$ is a complete graph and $F$ is a graph of high chromatic number, when both graphs are odd cycles, and also obtained bounds for $g(n,K_3,C_4)$.

	\subsection{Connection to other games}
	
	Let us now consider similarities and differences between Constructor-Blocker games and other types of games. Constructor-Blocker games are part of the broader category of \emph{positional games} (for a more comprehensive overview see e.g.~\cite{hefetz_positional_2014}), which are perfect information games between two players, who alternately claim elements of a finite set $X$ (called the \emph{board} of the game) equipped with a family $\mathcal{W}\subset \mathcal{P}(X)$ of \emph{winning sets}. The two most common settings are \emph{Maker-Maker games} and \emph{Maker-Breaker games}. In the Maker-Maker convention, the winner is the first player completing a winning set. In the Maker-Breaker convention, Maker wins if she is able to claim all of the elements of a winning set at the end of the game, Breaker wins otherwise. This also means, that in Maker-Maker games there can be a draw, but in Maker-Breaker games there has to be a winner.
	
	As an example for a positional game, we can consider the \emph{Maker-Breaker clique game}. Here, the board is the edge set of a complete graph ($X =E(K_n)$) and the winning sets are the edge sets of all $k$-cliques. Thanks to Ramsey's theorem, we know that, if $n$ is large enough, there must be a $k$-clique either in Maker's or in Breaker's graph at the end of the game. We can deduce, that Maker (since she is the first player) has a winning strategy, thanks to the concept of \emph{strategy stealing}. Indeed, one of the two players has to have a winning strategy, and if Breaker would have a winning strategy then Maker could apply this strategy (after playing her first move arbitrarily) and thus win. Therefore, we know that for $n$ large enough this game is a Maker's win. To analyse this game further, we can now consider either the smallest $n$ such that this game is a Breaker's win, or the largest number of $k$-cliques Maker can actually make.
	
	\medskip
	
	Going in the second direction, Bagan et al.~\cite{bagan_incidence_2024} introduced \emph{scoring positional games}. The two players alternately claim elements of the board until there are none left. In the Maker-Maker convention, the players get  a point whenever they complete a winning set and the score of the game is the difference between the number of points of the first and the second player. In the Maker-Breaker convention, the score is the number of winning sets completed by Maker. During the course of the game, Maker wants to maximize the number of winning sets, while Breaker wants to minimize it.
	
	\medskip
	
	We can relate Constructor-Blocker games very well to scoring positional games. Constructor takes over the role of Maker, Blocker takes over the role of Breaker, the board is the edge set of the complete graph $K_n$, and the winning sets are the subgraphs of $K_n$ isomorphic to $H$. But we also have to take into account that Constructor has further restrictions on how she is allowed to claim edges, since she has to avoid claiming a graph isomorphic to $F$.
	
	\medskip
	
	Let us also mention two other game settings that can transcribe this aspect of the Constructor-Blocker game -- that is, forbidding some subgraphs -- very well, but that do not take the maximization of the number of copies of $H$ into account.
	
	The first one we want to mention are \emph{Avoider-Enforcer games}. In an Avoider-Enforcer game, instead of winning sets we have a collection of losing sets. Avoider wins if she does not complete a losing set during the course of the game, Enforcer wins otherwise. Now let us look at a Constructor-Blocker game with $H = K_2$ and $F$ being some graph, then Constructor wants to maximize the number of edges in her graph, while at the same time not creating $F$ as a subgraph. The score of the game actually corresponds to the length of an Avoider-Enforcer game (where the losing sets are all subgraphs of $K_n$ isomorphic to $F$) before Avoider loses. As a sidenote, let us also mention, that this is actually the method how Balogh and Martin~\cite{Balogh_Avoider} determined $g(n,K_2,F)$ before Constructor-Blocker games were even introduced.
	
	The second game we want to mention are \emph{saturation games}. In a saturation game, both players alternately claim edges of a graph, but these edges do not belong to a specific player. Instead, both players are not allowed to claim some subgraph with the union of their edges. The first player wants to maximize the length of the game -- that is, the number of claimed edges when no player can claim any more edges -- while the second player wants to minimize it (see~\cite{Seress_1992} for an example).

	\subsection{Our results}
	
	We further extend the results in~\cite{balogh_constructor-blocker_2024} and~\cite{patkos_constructor-blocker_2022}, mainly focusing on the asymptotics of $g(n,K_3,\F)$ for several choices of $\F$ and comparing it to $ex(n,K_3,\F)$. Our first result considers the case where we relax the constraints on Constructor by not forbidding any subgraph at all (while maximizing the number of triangles). This setting actually corresponds to a regular scoring positional game. We get the following result: 
	
	
	\begin{thm}\label{thm : g(n,K_3,emptyset)}
		We have
		$$
		g(n,K_3,\emptyset) = (1+o(1)) \frac{n^3}{48} \, .
		$$
	\end{thm}
	
	Next, we consider the number of triangles while forbidding short paths. In \cite{patkos_constructor-blocker_2022}, Patk\'os, Stojakovi\'c, and Vizer considered $g(n,K_3,P_5)$ and showed that $g(n,K_3,P_5) = \frac{n}{4} - o(n)$. We extend the length of the forbidden path by one and get the following result:
	
	%
	%
	
	\begin{thm}\label{thm : g(n,K_3,P_6)}
		We have
		$$
		g(n,K_3,P_6) = (1 + o(1)) \frac{n}{2} \, .
		$$
	\end{thm}
	
	We also consider the cases, where instead of a path we forbid Constructor from claiming small stars (basically restricting the maximum degree in her graph). For the case that $F = S_4$, we get the following result:
	
	\begin{thm}\label{thm : g(n,K_3,S_4)}
		We have
		$$
		g(n,K_3,S_4) = (1+o(1)) \frac{n}{3} \, .
		$$
	\end{thm}
	
	We also considered the case $F = S_5$, but for this case the bounds that we obtained are not tight. Here, our result is the following:
	
	\begin{thm}\label{thm : g(n,K_3,S_5)}
		It holds that
		$$
		(1+ o(1)) \frac{2n}{3} \leq g(n,K_3,S_5) \leq (1+o(1)) n \, .
		$$
	\end{thm}
	
	Both of these theorems result from analysing all possible cases, which unfortunately did not help us to generalize our results to larger stars. However, we present an idea for a lower bound on $g(n,K_3,S_k)$ when $k$ goes to infinity (while growing not too fast):
	
	\begin{thm}\label{thm : g(n,K_3,S_k)}
		If $k(n) \underset{n\to \infty}{\longrightarrow}\infty$ and $k(n) = o(\sqrt{n})$, then 
		$$g
		(n,K_3,S_{k(n)+1}) \geq (1+o(1)) \frac{n\left(k(n)\right)^2}{48} \, .
		$$
	\end{thm}

	\medskip

	Lastly, we wonder what would happen in the Constructor-Blocker game if, instead of forbidding some subgraph in Constructor's graph, we would require it to remain planar. As bounded-degreeness (corresponding to forbidding some stars), planarity is among the classical properties studied in a graph. A graph is \textit{planar} if it can be drawn on a plane without any edges crossing. For a graph to be planar it cannot have certain graphs as a subgraph (for instance a $K_5$), but planarity cannot be entirely described by the absence of a finitite family of subgraphs. Therefore, this question cannot be stated in terms of Constructor-Blocker games as defined earlier, and we need to extend our framework. More precisely, we will define two new games: the \emph{planar Constructor-Blocker game} (PCB) and the \emph{embedded Constructor-Blocker game} (ECB).
	
	\medskip
	
	In the planar Constructor-Blocker game, Constructor can only claim edges such that her graph remains planar, while Blocker has no further restrictions. This means, that at any time of the game, there has to exist an embedding of $\C$ into the plane with no crossing edges. The embedding of Constructor's graph into the plane is not fixed at the beginning of the game. Similar to before, we will assume optimal play by both players, and the score of the game remains the number of triangles in $\C$ at the end of the game, which we will denote by $g_{\text{PCB}}(n,K_3)$.
	
	\medskip
	
	We obtained that Constructor can asymptotically build as many triangles as possible under the planarity constraint.
	
	\begin{thm}\label{thm : PCB}
		We have
		$$
		g_{\mathrm{PCB}}(n,K_3) = (1+o(1)) 3n \, .
		$$
	\end{thm}
	
	\medskip
	
	In the embedded Constructor-Blocker game, we impose even more restrictions on Constructor. When the game starts, we already fix an embedding of $K_n$ in the plane -- actually, we place the $n$ vertices as the $n^{th}$ roots of unity in the complex plane. Constructor is only allowed to claim edges inside of the unit disk, and her edges are not allowed to cross. She still aims to maximize the number of triangles in her graph. We will denote the score of this game by $g_{\text{ECB}}(n,K_3)$.
	
	\medskip
	
	For the embedded Constructor-Blocker game, we were able to provide both a lower and a non-trivial upper bound, but those do not match.
	
	\begin{thm}	\label{thm : ECB}
		It holds that
		$$
		(1+o(1))\frac{n}{2} \leq g_{\mathrm{ECB}}(n,K_3) \leq (1+o(1))\frac{2n}{3} \, .
		$$
	\end{thm}

	\subsection{Organization of the paper}
	
	In Section \ref{sec : prelim}, notations and preliminary results are given that will be useful later in the paper. 
	We prove the results concerning the Constructor-Blocker game in Section \ref{sec : results CB}, starting with forbidding nothing, then a path and finally a star. Theorems \ref{thm : g(n,K_3,emptyset)} to \ref{thm : g(n,K_3,S_k)} are proven in this section. Section \ref{sec : results planarity} is devoted to PCB and ECB games and contain the proofs of Theorem~\ref{thm : PCB} and Theorem~\ref{thm : ECB}.

	\section{Preliminaries}\label{sec : prelim}

	\subsection{Notation}
	
	
	Most of the notation used in this paper is standard 
	and is chosen according to~\cite{west2001introduction}.
	We set $[n]:=\{k\in\mathbb{N}:~ 1\leq k\leq n\}$ for
	every positive integer $n$. 
	Let $G$ be a graph. Then we write $V(G)$ and $E(G)$ for the \emph{vertex set} and the \emph{edge set} of $G$, respectively. 
	Let $v,w$ be vertices in $G$.
	If $\{v,w\}$ is an edge in $G$,  
	we write $vw$ for short. 
	The \emph{neighbourhood} of $v$ is
	$N_G(v) : =\{w\in V(G): vw\in E(G)\}$,
	and we call its size the \emph{degree} of $v$.
	Moreover, the \emph{minimum degree} in $G$ is denoted $\delta(G)$, and the \emph{maximum degree} in $G$ is denoted $\Delta(G)$.
	Often, when the graph $G$ is clear from the context,
	we omit the subscript $G$ in all definitions above.
	
	Let $G$ and $H$ be graphs. 
	We say that $H$ is a \emph{subgraph} of $G$,
	denoted $H\subset G$, if both $V(H)\subset V(G)$ and 
	$E(H)\subset E(G)$ hold.
	The number of copies of $H$ in $G$ is the number of subgraphs of $G$ isomorphic to $H$.
	For a subset $A\subset V(G)$ we let $G-A$
	be the subgraph of $G$ 
	with vertex set $V(G)\setminus A$
	and edge set $\{vw\in E(G):~ v,w\in V(G)\setminus A\}$.
	
	We write $P_n$ for the \emph{path} on $n$ vertices, i.e.~with vertex set $V(P_n)=\{v_1,\ldots,v_n\}$
	and edge set $E(P_n)=\{v_iv_{i+1}:~ i\in [n-1]\}$. Additionally, we call a $P_n$ with one additional vertex that is connected to all $n$ vertices of $P_n$ a \emph{fan} of size $n$, denoted by $F_n$.
	Similarly,
	we write $C_n$ for the \emph{cycle} on $n$ vertices, i.e.~with vertex set $V(C_n)=\{v_1,\ldots,v_n\}$
	and edge set $E(C_n)=\{v_iv_{i+1}:~ i\in [n-1]\}\cup \{v_nv_1\}$, and call a $C_n$ with one additional vertex that is connected to all $n$ vertices of $C_n$ a \emph{wheel} of size $n$, denoted by $W_n$.
	We write $K_n$ for the \emph{complete graph} on $n$ vertices. In particular, $K_3$ is a triangle. A \emph{triangle with a pendant leg} is a triangle with one additional edge attached to one of its vertices.
	We denote a \emph{star} with $n$ leaves as $S_n$. Additionally, we call an $S_2$ a \emph{cherry} and call the central vertex its root.
	
	For functions $f,g:\mathbb{N} \rightarrow \mathbb{R}$, we write $f(n)=o(g(n))$ if $\lim_{n\rightarrow \infty} \left|\frac{f(n)}{g(n)}\right| = 0$ and $f(n) = O(g(n))$ if $\limsup_{n\rightarrow \infty} \frac{\left|f(n)\right|}{g(n)} < \infty$.
	
	Assume a Constructor-Blocker game is in progress, then we will denote by $\C$ (resp. $\B$) \emph{Constructor's} (resp. \emph{Blocker's}) \emph{graph}, that is the graph consisting of the edges claimed by Constructor (resp. Blocker). If an edge has not been claimed by either player, we call such an edge \emph{free}.

	\subsection{Relevant game tools.}
	
	The Erd\H{o}s-Selfridge Criterion~\cite{ERDOS1973298} is a classical result for Maker-Breaker games, which (if applicable) gives an easy argument for a win by Breaker. It is also often used in cases where Maker takes over the role of Breaker and proceeds to claim at least one element of every winning set. It was stated in its original form as follows:
	
	\begin{lem}[Erd\H{o}s-Selfridge-Criterion~\cite{ERDOS1973298}]
		Let $X$ be a finite set and let $\mathcal{W}\subseteq \mathcal{P}(X)$ satisfying
		$$
		\sum_{W \in \mathcal W} 2^{-|W| + 1} < 1 \, .
		$$
		Then in the Maker-Breaker game on board $X$ with winning sets $\mathcal{W}$, Breaker has a strategy to claim at least one element in each of the winning sets in $\mathcal{W}$.
	\end{lem}

	
	In~\cite{bagan_incidence_2024}, this criterion was adapted for scoring positional games, and we state here a version corresponding to the players claiming edges of a graph.

	\begin{thm}[Theorem 10 in~\cite{bagan_incidence_2024}]\label{thm : ES}
		Let $G = (V,E)$ be a graph on $n$ vertices with $m$ edges. Let $\mathcal{W} \subset \mathcal{P}(E)$ and consider the scoring Maker-Breaker positional game on the board $X=E$ in which the winning sets are the elements of $\mathcal{W}$. Set $\ell:= \underset{e\neq e'\in E}{\max}|\{W \in \mathcal{W}\,|\, e,e' \in W\}|$.
		If Maker starts, then the score of the game is at least $\sum_{W \in \mathcal{W}}2^{-|W|}-\frac{m\ell}{8}$.
		If Breaker starts, then the score is at most $\sum_{W \in \mathcal{W}}2^{-|W|}$.
	\end{thm}
	
	
	Let us also prove a helpful lemma, which provides Constructor with a strategy to claim a triangle with a pendant leg while keeping the maximum degree in her graph at $3$. We prove this lemma here, since we will use it on two different occasions in Section~\ref{subsec : results K_3-S} and Section~\ref{subsec : pcb}.
	
	\begin{lem}
		\label{lemma : pendant leg}
		If $v_1,\dots,v_5$ are isolated in $\C$ and if there is no Blocker edge between these vertices, then Constructor can build a triangle with a pendant leg within four moves (if this subgraph is not forbidden).
	\end{lem}
	
	\begin{proof}
		We give a strategy for Constructor to achieve this goal. In her first move she plays $v_1v_2$. Blocker might play an edge between two vertices among $v_1,\dots,v_5$, without loss of generality assume this edge to be adjacent to $v_5$. Then Constructor plays $v_1v_3$.
		\begin{itemize}
			\item If Blocker plays $v_2v_3$, then Constructor plays $v_1v_4$. Blocker cannot prevent her from creating a triangle. Indeed, he might play $v_2v_4$ or $v_3v_4$, she will play the other one.
			\item Otherwise Constructor plays $v_2v_3$. Then even after Blocker's move she will be able to connect $v_4$ or $v_5$ to one of $v_1,v_2,v_3$.
		\end{itemize}
		An example is given in Figure \ref{fig : pendant leg strategy}.
	\end{proof}
	
	\begin{figure}[h!]
		\centering
		\input{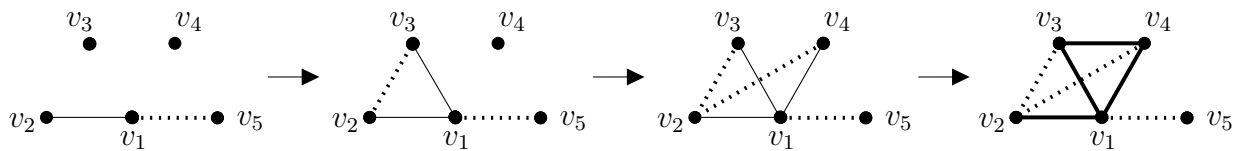}
		\caption{How to build a triangle with a pendant leg? An example of Constructor's (solid lines) and Blocker's (dotted lines) successive moves, Constructor being the first player.}
		\label{fig : pendant leg strategy}
	\end{figure}
	
	The idea behind this lemma is that Constructor cannot build a triangle on her first trial, but she can create many threats so that Blocker will not be able to handle all of them, thus Constructor will still be able to create triangles.
	
	\medskip
	
	Lastly, we will also prove a helpful lemma for Blocker, which helps to bound the number of Constructor's edges inside each of her components. Blocker can achieve this goal by answering in the component where Constructor played the move just before. It implies that, sometimes, following this strategy is enough for Blocker to prevent Constructor from achieving her goal, and we will use it later on.
	
	\begin{lem}\label{lemma : play in the same cc}
		Assume that Blocker always answers a move by Constructor in the same connected component of $\C$ if possible. If at the end of the game some connected component $C$ of $\C$ contains $r$ vertices, then at least $\frac{\binom{r}{2}-\lfloor \frac{r}{2}\rfloor}{2}$ edges between vertices of $C$ are not claimed by Constructor.
	\end{lem}
	
	\begin{proof}
		We proceed by induction on $r$.
		
		For $r\in \{1,2\}$, the bound is 0, so this is trivially true. 
		
		Let $r\geq 3$ and assume this to be true for any $s<r$. Let $C$ be a connected component of $\C$ of size $r$. Look at the time when Constructor connects two disjoint connected components $C_1$ and $C_2$ of sizes $1 \leq r_1,r_2 < r$ with $C_1 \cup C_2 = C$ (and $r_1+r_2=r$). Up to this point, Blocker has played in $C_1$ whenever Constructor has done so, respectively in $C_2$. If Constructor plays to maximize the number of edges in $C$, without loss of generality, we can assume that from now on she plays in $C_1$ (resp. in $C_2$, resp. an edge between $C_1$ and $C_2$) whenever Blocker does so.
		
		This ensures that, in total, at least $ \frac{\binom{r_1}{2}-\left\lfloor \frac{r_1}{2}\right\rfloor}{2} + \frac{\binom{r_2}{2}-\left\lfloor \frac{r_2}{2}\right\rfloor}{2} + \left\lfloor \frac{r_1r_2}{2}\right\rfloor$ edges are not claimed by Constructor.
		
		If $r_1 = 2k_1$ and $r_2 = 2k_2$ with $k_1,k_2$ positive integers, then this gives
		$ k_1(k_1-1) + k_2(k_2-1) + 2k_1k_2 = (k_1+k_2)(k_1+k_2-1) =  \frac{\binom{r}{2}-\left\lfloor \frac{r}{2}\right\rfloor}{2}$.\\
		The other possibilities for the parities of $r_1$ and $r_2$ give the same result.
		
		Hence by induction this lemma is true for any $r\geq 1$.
	\end{proof}

	\section{Results on the Constructor-Blocker game}\label{sec : results CB}
	
	In this section we will give proofs for the cases when Constructor has no restrictions (Theorem~\ref{thm : g(n,K_3,emptyset)}), when Constructor is not allowed to claim short paths (Theorem~\ref{thm : g(n,K_3,P_6)}), and when Constructor is not allowed to claim small stars (Theorem~\ref{thm : g(n,K_3,S_4)} and Theorem~\ref{thm : g(n,K_3,S_5)}) or large stars (Theorem~\ref{thm : g(n,K_3,S_k)}).
	
	\subsection{Not forbidding anything}\label{subsec : res K_3-empty}
	In this subsection we will prove Theorem~\ref{thm : g(n,K_3,emptyset)}. In this setting, both players are allowed to claim edges without restrictions, and at the end of the game we count the number of triangles in Constructor's graph. This setting actually corresponds to a scoring positional game. The board is the edge set of a complete graph, and the winning sets are all triangles. Before we prove our result, we will first compute the maximum number of triangles a graph on $n$ vertices can have.
	
	\begin{prop}
		The maximum number of triangles in an $n$-vertex graph is
		$$
		ex(n,K_3,\emptyset) = \binom{n}{3} = (1+o(1))\frac{n^3}{6} \, .
		$$
	\end{prop}
	
	This result comes from the fact that computing the number of triangles in a graph is exactly computing the number of sets of three different vertices $u,v,w$ such that all the edges $uv,uw,$ and $vw$ are present in the graph. Thus there are at most $\binom{n}{3}$ triangles in a graph on $n$ vertices, and there is equality if and only if the graph is complete.
	
	In the game version, Constructor will not be able to claim all the edges of the complete graph -- she will basically claim half of them. Therefore, maybe a more relevant question would be the maximum number of triangles that a graph on $n$ vertices with at most half of the edges can have. This was answered by Rivin in~\cite{RIVIN2002647} who found that this number is (asymptotically) $\frac{n^3}{12\sqrt{2}}$ and is achieved by a clique of size $\frac{n}{\sqrt{2}}$. In the Constructor-Blocker game, we obtain the same order of magnitude, but our constant factor is smaller.

	\begin{proof}[Proof of Theorem \ref{thm : g(n,K_3,emptyset)}]
		In order to prove our result, we make use of Theorem~\ref{thm : ES}. We apply this theorem to the complete graph $K_n$, with $\mathcal{W}$ be the set of all triangles, Constructor playing as Maker and Blocker playing as Breaker. Since a fixed pair of edges can only be in at most one triangle, $\ell=1$. Constructor starts, thus
		$$
		g(n,K_3,\emptyset) \geq \binom{n}{3}\cdot 2^{-3} -\frac{n(n-1)}{2\cdot 8} = (1+o(1))\frac{n^3}{48} \, .
		$$	
		In order to get an upper bound, assume temporarily that Blocker starts. Theorem \ref{thm : ES} states that the score of the game is at most $\binom{n}{3}\cdot 2^{-3} = (1+o(1)) \frac{n^3}{48}$. But the edge Blocker claimed at the beginning blocks at most $n-2$ triangles, which is negligible with respect to $n^3$. Thus this minor change on the first player does not affect our asymptotic equivalent and therefore we have
		$$
		g(n,K_3,\emptyset) =  (1+o(1)) \frac{n^3}{48} \, .
		$$
	\end{proof}
	
	\medskip
	
	Note that $g(n,K_3,\emptyset)$ is the expected number of triangles if both players play randomly during the whole game -- saying it differently, the number of triangles in Constructor's graph at the end of the game is the expected number of triangles in a uniform random graph.

	\medskip
	
	\subsection{Forbidding a path}\label{subsec : results K_3-P}
	
	Next, we turn to the case where we forbid Constructor from claiming a (short) path. Before considering the game, we give the extremal numbers when forbidding a path:
	
	\begin{prop}[Corollary 1.7 in~\cite{luo_maximum_2018}]\label{luo_maximum}
		For $n \geq k \geq 4$ and $s \geq 2$ the following holds:
		$$
		ex(n,K_s,P_k) = \frac{n}{k-1}\cdot \binom{k-1}{s} \, .
		$$
	\end{prop}
	
	In particular, this tells us that $ex(n,K_3,P_4) = \frac{n}{3}$ 
	
	Turning to games, since $P_3$ is a subgraph of $K_3$, we will start looking at $P_k$ for $k=4$. In that case, Blocker can prevent Constructor from building a triangle.
	
	\begin{prop}
		We have
		$$
		g(n,K_3,P_4) = 0 \, .
		$$
	\end{prop}
	
	\begin{proof}
		The only possibility for a $K_3$ to appear in a $P_4$-free graph is as an isolated triangle. Blocker can easily prevent such isolated triangles with the following strategy: whenever an isolated $P_3$ appears in $\C$, Blocker plays the edge that would complete a triangle, otherwise Blocker plays arbitrarily. Therefore $g(n,K_3,P_4) = 0$.
	\end{proof}
	
	

	Now let us turn to $P_6$. This case is proved similarly to the case of $P_5$, which was proved by Patk\'os, Stojakovi\'c, and Vizier in \cite{patkos_constructor-blocker_2022} by looking at the possibilities of how components of a $P_5$-free graph could look like. For our case, these configurations are shown in Figure~\ref{fig : K_3-P_6 : cc}. Either there exists an edge which is in several triangles (case (a)) -- Blocker will have to be careful about this situation and basically make sure that there are some vertices of the component that are not in any triangle to balance -- or every edge is in at most one triangle (case (b)). In the latter case, only one vertex of the component can be in several triangles and the number of triangles is a bit less than half of the number of vertices of the component.
	
	\begin{figure}[h!]
		\centering
		\input{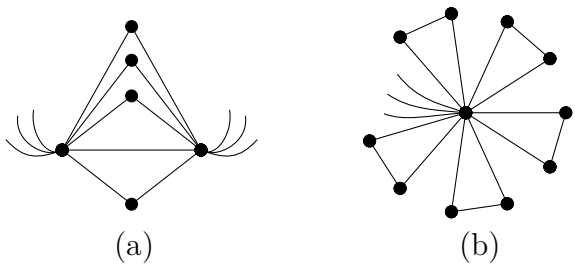}
		\caption{How connected components with many triangles in a $P_6$-free graph can look like. The curved lines mean that there might or might not be such edges.}
		\label{fig : K_3-P_6 : cc}
	\end{figure}

	\begin{proof}[Proof of Theorem \ref{thm : g(n,K_3,P_6)}]
		We will start with a strategy for Constructor. Her strategy will be to create components of case (b), i.e.~she will create components of many triangles intersecting in a single vertex.
		We will show that with this strategy for large enough $n$ Constructor will claim at least $\frac{n}{2}(1-\frac{2}{\ln(n)}) = (1+o(1))\frac{n}{2}$ triangles. Her strategy is split into two stages.
		
		\medskip
		
		{\bf Stage I:} While there are more than $\frac{n}{\ln(n)}$ isolated vertices in $\C$, Constructor constructs big stars. She starts by picking an arbitrary vertex $v \in G$ and then afterwards she claims every available edge incident to $v$. Once there are no more edges available, Constructor picks an isolated vertex in $\C$ that has the lowest degree in $\B$ and repeats this process, until the number of isolated vertices in $\C$ is at most $\frac{n}{\ln(n)}$. Then she proceeds to Stage II.
		
		\medskip
		
		{\bf Stage II:} At the end of Stage I, Constructor's graph will consist of stars. During this stage, for each star Constructor creates a matching of the leaves by claiming arbitrary edges that extend her matching.
		
		Once there are no more free edges available that would extend her matching, she stops playing.

		\medskip
		
		These stages are illustrated in Figure \ref{fig : K_3-P_6 : C strat}.
		
		\begin{figure}[h!]
			\centering
			\input{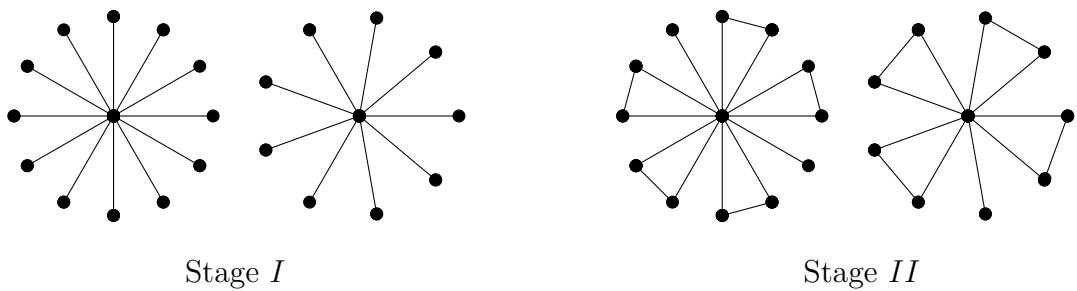}
			\caption{Constructor's strategy when forbidding $P_6$}
			\label{fig : K_3-P_6 : C strat}
		\end{figure}
		
		\medskip
		
		{\bf Strategy discussion:}
		
		We will first look at Stage I and show that all of the stars built by Constructor will have almost linear size.
		
		\begin{claim}
			If Constructor follows her strategy, then, for $n$ large enough, at the end of the first stage, every star in $\C$ will have size at least $\frac{n}{3\ln(n)}$.
		\end{claim}
		
		\begin{claimproof}
			Stage I lasts as most $(1-\frac{1}{\ln(n)})n$ turns, therefore Constructor and Blocker each will have played at most $(1-\frac{1}{\ln(n)})n$ edges at the end of Stage I. Thus, while there are more than $\frac{n}{\ln(n)}$ isolated vertices in $\C$, there is at least one vertex $v \in \B \setminus \C$ with $d_{\B \setminus \C}(v) < 2\ln(n)$. This means that all of the stars Constructor builds will have at least $\frac{1}{2}\left(\frac{n}{\ln(n)} -2\ln(n)\right) \geq \frac{n}{3\ln(n)}$ leaves for $n$ large enough.
		\end{claimproof}
		
		\medskip
		
		This claim also tells us, that there are at most $3\ln(n)$ star centers in Constructor's graph. Now, we argue that the matchings that Constructor will build during Stage II will cover most of the leaves. For each star, Constructor can match all of the leaves except at most $2\sqrt{n}$ of them. Indeed, if Constructor cannot add an edge between the $k$ remaining leaves of a star, it means that Blocker has blocked all of the $\binom{k}{2}$ edges, but during the whole game Blocker claims less than $\frac{3n}{2}$ edges. We conclude that $\binom{k}{2} < \frac{3n}{2}$ has to hold, and for $n$ large enough we have $k\leq 2\sqrt{n}$.
		
		\medskip
		
		Finally, let us count the number of triangles in $\C$ at the end of the game. There are at most $\frac{n}{\ln(n)}$ isolated vertices and $\frac{(1-\frac{1}{\ln(n)})n\cdot 3\ln(n)}{n} \cdot 2\sqrt{n} < 3\ln(n) \cdot 2\sqrt{n}$ vertices that are leaves of stars but not in the matchings. Then there is one triangle every two vertices that are both leaves of a star and in the matching. This gives us at least $\frac{1}{2}(n-\frac{n}{\ln(n)} - 3\ln(n) - 3\ln(n) \cdot 2\sqrt{n}) \leq \frac{n}{2}(1- \frac{2}{\ln(n)}) = (1+o(1))\frac{n}{2}$ triangles.
		
		\medskip
		
		We now turn to the upper bound and provide a strategy for Blocker ensuring that Constructor does not create more than $\frac{n}{2}$ triangles.
		
		If possible, Blocker always plays in the same connected component $C$ in which Constructor just played, according to the following rules:
		\begin{enumerate}
			\item If there exists an edge $v_1v_2$ which is in two triangles in $C$:
			\begin{itemize}
				\item If Constructor plays $vv_i$ ($i \in \{1,2\}$) for some $v \in G$ and if $vv_{3-i}$ is free, then Blocker claims $vv_{3-i}$.
				\item Otherwise, if for some $v \in \C$ the edge $vv_i$ ($i \in \{1,2\}$) is in $\C$ and if $vv_{3-i}$ is free, then Blocker claims $vv_{3-i}$.
				\item If none of the two previous cases apply, Blocker plays arbitrarily.
			\end{itemize}
			This corresponds to Figure \ref{fig : K_3-P_6 : cc}.a.
			\item Otherwise, if Constructor completes the first $C_4$ of $C$ and at least one diagonal is free, then Blocker plays this diagonal, prioritizing the diagonal with a vertex of largest Constructor degree out of the four vertices.
			\item Otherwise, if there is a free edge that could close a triangle in $C$, then Blocker takes it, prioritizing an edge adjacent to the edge Constructor just played.
			\item Otherwise Blocker plays an arbitrary edge.
		\end{enumerate}
		
		We show that this strategy prevents Constructor from creating too many triangles. We will start with the following claim:
		
		\begin{claim}
			At any time of the game, the number of triangles in a connected component $C$ of $\C$ is at most $\frac{|C|}{2}$.
		\end{claim}
		
		\begin{claimproof}
			We first deal with the small connected components. Let us assume, that Blocker plays according to the strategy in Lemma~\ref{lemma : play in the same cc}. Since Blocker always answers in the same connected component as the one  Constructor just played an edge in (when possible), the following holds:
			\begin{itemize}
				\item If $|C| = 3$, then $C$ contains no triangle.
				\item If $|C| = 4$, then Constructor has at most four edges, thus at most $1$ triangle.
				\item If $|C| = 5$, then Constructor has at most six edges, thus at most two triangles (this can be shown by an easy case distinction).
			\end{itemize}
			
			Now assume that $|C| \geq 6$ and assume that there is an edge $v_1v_2$ that is in several triangles. The other vertices of $C$ will be denoted by $u_1,u_2,\dots$ (see Figure \ref{fig : K_3-P_6 : B strat}). In order to have at least three triangles without any $P_6$, Constructor cannot have any edge between some $u_i$ and $u_j$ for $i\neq j$, thus every $u_i$ is connected to $v_1,v_2$, or both. Let us look back at the time when the edge $v_1v_2$ was claimed by Constructor.
			
			\begin{itemize}
				\item Either it was the first edge of $C$ played by Constructor. Then when Constructor played another edge in $C$ (some $v_1u_i$ ou $v_2u_i$), Blocker played (or had already played) $v_2u_i$ or $v_1u_i$.
				\item Or there was already a $C_4$ in $C$, say $v_1u_1v_2u_2$ was the first one. Let us show that this is actually impossible. When this cycle was completed by Constructor, one of the following happened:
				\begin{itemize}
					\item Either $v_1$ and $v_2$ both had degree 2. Then Blocker already had claimed one of the diagonals of the $C_4$, and claimed the other one, thus $v_1v_2 \in \B$.
					\item Or $v_1$ or $v_2$ had degree at least 3. Then, because of Rule (2) in her strategy Blocker blocked $v_1v_2$.
				\end{itemize}
				Thus it is impossible that Constructor completed a cycle in $C$ before claiming $v_1v_2$.
				\item Or there was at least one edge $v_1u_i$ or $v_2u_i$ and no cycle. When Constructor played $v_1v_2$, then because of Rule (3) in her strategy Blocker played an edge of the form $v_2u_i$ or $v_1u_i$.
			\end{itemize}
			
			\begin{figure}[h!]
				\centering
				\begin{tikzpicture}[x=0.75pt,y=0.75pt,yscale=-1,xscale=1]

\draw [color={rgb, 255:red, 0; green, 0; blue, 0 }  ,draw opacity=1 ]   (134.95,101.26) -- (89.3,19) ;
\draw    (43.65,101.26) -- (89.3,19) ;
\draw    (43.65,101.26) -- (89.3,65.19) ;
\draw    (134.95,101.26) -- (89.3,65.19) ;
\draw    (89.3,137.33) -- (134.95,101.26) ;
\draw    (89.3,137.33) -- (43.65,101.26) ;
\draw    (134.95,101.26) -- (43.65,101.26) ;
\draw    (124,146) -- (134.95,101.26) ;
\draw    (58,22) -- (43.65,101.26) ;
\draw [draw opacity=0]   (89.3,137.33) -- (124,146) ;
\draw [draw opacity=0]   (58,22) -- (89.3,19) ;

\draw (17,93) node [anchor=north west][inner sep=0.75pt]   [align=left] {$\displaystyle v_{1}$};
\draw (145,91) node [anchor=north west][inner sep=0.75pt]   [align=left] {$\displaystyle v_{2}$};
\draw (84,142) node [anchor=north west][inner sep=0.75pt]   [align=left] {$\displaystyle u_{1}$};
\draw (83,71) node [anchor=north west][inner sep=0.75pt]   [align=left] {$\displaystyle u_{2}$};
\draw (82,31) node [anchor=north west][inner sep=0.75pt]   [align=left] {$\displaystyle u_{3}$};
\draw (32,11) node [anchor=north west][inner sep=0.75pt]   [align=left] {$\displaystyle u_{4}$};
\draw (131,139) node [anchor=north west][inner sep=0.75pt]   [align=left] {$\displaystyle u_{5}$};

\draw [fill={rgb, 255:red, 0; green, 0; blue, 0 }  ,fill opacity=1 ]  (89.3, 19) circle [x radius= 3, y radius= 3]   ;
\draw [fill={rgb, 255:red, 0; green, 0; blue, 0 }  ,fill opacity=1 ]  (134.95, 101.26) circle [x radius= 3, y radius= 3]   ;
\draw [fill={rgb, 255:red, 0; green, 0; blue, 0 }  ,fill opacity=1 ]  (134.95, 101.26) circle [x radius= 3, y radius= 3]   ;
\draw [fill={rgb, 255:red, 0; green, 0; blue, 0 }  ,fill opacity=1 ]  (134.95, 101.26) circle [x radius= 3, y radius= 3]   ;
\draw [fill={rgb, 255:red, 0; green, 0; blue, 0 }  ,fill opacity=1 ]  (134.95, 101.26) circle [x radius= 3, y radius= 3]   ;
\draw [fill={rgb, 255:red, 0; green, 0; blue, 0 }  ,fill opacity=1 ]  (89.3, 19) circle [x radius= 3, y radius= 3]   ;
\draw [fill={rgb, 255:red, 0; green, 0; blue, 0 }  ,fill opacity=1 ]  (43.65, 101.26) circle [x radius= 3, y radius= 3]   ;
\draw [fill={rgb, 255:red, 0; green, 0; blue, 0 }  ,fill opacity=1 ]  (43.65, 101.26) circle [x radius= 3, y radius= 3]   ;
\draw [fill={rgb, 255:red, 0; green, 0; blue, 0 }  ,fill opacity=1 ]  (89.3, 137.33) circle [x radius= 3, y radius= 3]   ;
\draw [fill={rgb, 255:red, 0; green, 0; blue, 0 }  ,fill opacity=1 ]  (43.65, 101.26) circle [x radius= 3, y radius= 3]   ;
\draw [fill={rgb, 255:red, 0; green, 0; blue, 0 }  ,fill opacity=1 ]  (89.3, 137.33) circle [x radius= 3, y radius= 3]   ;
\draw [fill={rgb, 255:red, 0; green, 0; blue, 0 }  ,fill opacity=1 ]  (89.3, 65.19) circle [x radius= 3, y radius= 3]   ;
\draw [fill={rgb, 255:red, 0; green, 0; blue, 0 }  ,fill opacity=1 ]  (134.95, 101.26) circle [x radius= 3, y radius= 3]   ;
\draw [fill={rgb, 255:red, 0; green, 0; blue, 0 }  ,fill opacity=1 ]  (134.95, 101.26) circle [x radius= 3, y radius= 3]   ;
\draw [fill={rgb, 255:red, 0; green, 0; blue, 0 }  ,fill opacity=1 ]  (134.95, 101.26) circle [x radius= 3, y radius= 3]   ;
\draw [fill={rgb, 255:red, 0; green, 0; blue, 0 }  ,fill opacity=1 ]  (134.95, 101.26) circle [x radius= 3, y radius= 3]   ;
\draw [fill={rgb, 255:red, 0; green, 0; blue, 0 }  ,fill opacity=1 ]  (134.95, 101.26) circle [x radius= 3, y radius= 3]   ;
\draw [fill={rgb, 255:red, 0; green, 0; blue, 0 }  ,fill opacity=1 ]  (124, 146) circle [x radius= 3, y radius= 3]   ;
\draw [fill={rgb, 255:red, 0; green, 0; blue, 0 }  ,fill opacity=1 ]  (43.65, 101.26) circle [x radius= 3, y radius= 3]   ;
\draw [fill={rgb, 255:red, 0; green, 0; blue, 0 }  ,fill opacity=1 ]  (43.65, 101.26) circle [x radius= 3, y radius= 3]   ;
\draw [fill={rgb, 255:red, 0; green, 0; blue, 0 }  ,fill opacity=1 ]  (43.65, 101.26) circle [x radius= 3, y radius= 3]   ;
\draw [fill={rgb, 255:red, 0; green, 0; blue, 0 }  ,fill opacity=1 ]  (58, 22) circle [x radius= 3, y radius= 3]   ;
\end{tikzpicture}
				\caption{Notations for the proof of Theorem \ref{thm : g(n,K_3,P_6)}}
				\label{fig : K_3-P_6 : B strat}
			\end{figure}
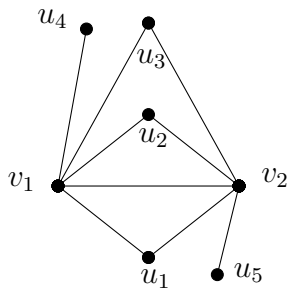

			In all cases, we made sure that when $v_1v_2$ becomes part of two triangles, some edge $v_1u_i$ or $v_2u_i$ is blocked. Then Blocker plays according to Rule (1). He will prevent at least half of the remaining vertices $u_j$ from being in any triangle (that makes $\lfloor\frac{|C|-3}{2}\rfloor$ vertices). In the end there cannot be more than $\lfloor\frac{|C|}{2}\rfloor$ vertices $u_j$ in a triangle and thus $C$ contains at most $\lfloor\frac{|C|}{2}\rfloor $ triangles.
			
			Note that if in some $C$ no edge is in several triangles, then at most one vertex of $C$ can be in several triangles (see Figure \ref{fig : K_3-P_6 : cc}.b), otherwise there would be some $P_6$. Therefore, the number of triangles in such components cannot be more than $\frac{|C|-1}{2}$.
		\end{claimproof}
		
		\medskip
		
		If we now consider all of Constructor's components at the end of the game, if Blocker follows his strategy the number of triangles in $\C$ is at most $\frac{n}{2}$, which gives us our desired result.	
	\end{proof}

	\subsection{Forbidding a star}\label{subsec : results K_3-S}
	We now investigate the case where we forbid Constructor from claiming a star $S_k$ for small values of $k$. Note that for a graph being $S_k$-free is equivalent to being of maximum degree at most $k-1$.
	
	\medskip
	
	First, let us consider the maximum number of triangles a graph of a certain maximum degree can have. This was proven by Chase in~\cite{chase_maximum_2020}:
	
	\begin{prop}[Theorem 1 in~\cite{chase_maximum_2020}]
		For $n \geq 1$ and $k\geq 2$, the following holds:
		$$
		ex(n,K_3,S_k) = (1+o(1)) \frac{n}{k} \binom{k}{3} = (1+o(1))\frac{nk^2}{6} \, .
		$$
	\end{prop}
	
	The proof is done by induction on $n$ and by exhibiting a vertex whose neighbourhood does not intersect too many triangles.
	The extremal graphs are once again disjoint unions of cliques of size $k$, that contain around $\frac{n}{k} \binom{k}{3}$ triangles.
	
	\medskip
	
	Let us turn to the game version. The only way for Constructor to build a triangle while not exceeding maximum degree $2$ in her graph (i.e.~not claiming an $S_3$) is to build an isolated triangle, however Blocker can prevent this from happening, therefore $g(n,K_3,S_3) = 0$. Let us thus turn to the case of $S_4$ and prove Theorem~\ref{thm : g(n,K_3,S_4)} next.
	
	\subsubsection{Forbidding $S_4$}
	
	We will split this case into two parts, one focussing on a strategy for Constructor, and one for Blocker. Afterwards, Theorem~\ref{thm : g(n,K_3,S_4)} will trivially follow from both parts combined. We will start with Constructor's side to get the lower bound, and before we go into more details, we will briefly state the main ideas of the proof. We will show that Constructor can build a big chain of triangles connected by one edge (see Figure~\ref{fig : K_3-S_4 : C strat}, induction). The vertices in this chain are (almost) all in a triangle, and we show that the number of vertices that are not in this chain is negligible. We will state this as a new proposition:
	
	\begin{prop}\label{prop : lower bound g(n,K_3,S_4)}
		Constructor has a strategy to ensure that the following holds:
		$$
		g(n,K_3,S_4) \geq \frac{n}{3}(1+o(1)) \, .
		$$
	\end{prop}
	
	\begin{proof}[Proof of Proposition~\ref{prop : lower bound g(n,K_3,S_4)}]	
		We provide a strategy for Constructor to build at least $\frac{n}{3}-o(n)$ triangles. This strategy consists of two stages.
		
		\medskip
		
		{\bf Stage I:} During this stage, Constructor creates many many isolated cherries. While there are at least $3\sqrt{n}$ isolated vertices in $\C$, Constructor picks 3 vertices $v_1,v_2,v_3$ that are isolated in $\C$ and are not connected in $\B$. She plays $v_1v_2$. If Blocker plays $v_1v_3$, then Constructor plays $v_2v_3$, otherwise she plays $v_1v_3$. Then $v_1,v_2,v_3$ form a new cherry.
		
		Once there are less than $3\sqrt{n}$ isolated vertices in Constructor's graph, she proceeds to Stage II.
		
		\medskip
		
		{\bf Stage II:} During this stage, Constructor will recursively construct a component $C_k$, which will be a connected component of $\C$ satisfying the following properties:
		\begin{itemize}
			\item $C_k$ contains exactly $k$ triangles.
			\item $C_k$ contains exactly $3k+1$ vertices.
			\item There is a vertex $x_k$ in $C_k$ of degree $1$, connected to at most $\sqrt{n}+2$ cherries in $\B$.
		\end{itemize}
		
		Once she is not able to follow this strategy any more, she stops playing. Set $k_0$ to be equal to $k$ at the final step.
		
		\medskip
		
		We will briefly sketch how Constructor will play during Stage II. Assume, that so far Constructor could follow her strategy, thus there already exist $C_k$ and $x_k$ as described above.
		While there is a cherry that is connected to less than $\sqrt{n}$ other cherries in $\B$ and that is not connected to $x_k$ in $\B$, Constructor picks this cherry to continue her chain. We will name its vertices $v_1,v_2,v_3$, where $v_1$ is the root of this cherry (meaning that $d_{\C}(v_1)=2$). Constructor plays $x_kv_1$. If Blocker plays $x_kv_2$, then Constructor plays $x_kv_3$, otherwise she plays $x_kv_2$. Let this new connected component be $C_{k+1}$ and $x_{k+1}$ be $v_2$ or $v_3$, depending on what Constructor just played.
		
		\medskip
		
		\emph{Remark.} She starts Stage II with $C_0$ being an isolated vertex in $\C$ (namely $x_0$) and then she picks a cherry which is connected to at most $\sqrt{n}+2$ cherries in $\B$ and is not connected to $x_0$, then she follows the instructions as above. The resulting component $C_1$ is actually a triangle with a pendant leg, and its construction corresponds to the one in Lemma \ref{lemma : pendant leg}.
		
		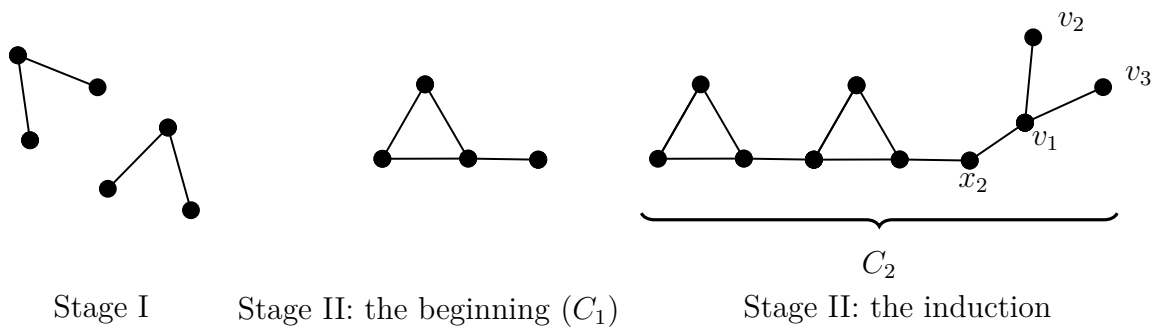
\begin{figure}[h!]
			\centering
			\tikzset{every picture/.style={line width=0.75pt}} 

\begin{tikzpicture}
	[x=0.75pt,y=0.75pt,yscale=-1,xscale=1]

	\draw    (45.56,35.33) -- (51.69,78.2) ;
	\draw    (85.75,51.46) -- (45.56,35.33) ;
	
	\draw    (121.24,71.79) -- (132.85,113.51) ;
	\draw    (90.92,102.7) -- (121.24,71.79) ;
	\draw  [color={rgb, 255:red, 0; green, 0; blue, 0 }  ,draw opacity=1 ] (272.68,87.45) -- (229.38,87.55) -- (250.94,50) -- cycle ;
	\draw    (272.68,87.45) -- (308,88) ;
	\draw [draw opacity=0]   (250.94,50) -- (308,88) ;
	\draw [draw opacity=0]   (308,88) -- (229.38,87.55) ;
	\draw  [color={rgb, 255:red, 0; green, 0; blue, 0 }  ,draw opacity=1 ] (411.68,87.45) -- (368.38,87.55) -- (389.94,50) -- cycle ;
	\draw    (411.68,87.45) -- (447,88) ;
	\draw  [color={rgb, 255:red, 0; green, 0; blue, 0 }  ,draw opacity=1 ] (490.3,87.89) -- (447,88) -- (468.56,50.45) -- cycle ;
	\draw [draw opacity=0]   (525.62,88.45) -- (447,88) ;
	\draw    (490.3,87.89) -- (525.62,88.45) ;
	\draw  [draw opacity=0] (592.89,51.39) -- (553.49,69.36) -- (557.63,26.25) -- cycle ;
	\draw    (525.62,88.45) -- (553.49,69.36) ;
	\draw    (553.49,69.36) -- (557.63,26.25) ;
	\draw    (553.49,69.36) -- (592.89,51.39) ;
	\draw    (389.94,50) -- (368.38,87.55) ;
	\draw    (468.56,50.45) -- (447,88) ;
	
	\draw (555.49,72.36) node [anchor=north west][inner sep=0.75pt]   [align=left] {$\displaystyle v_{1}$};
	\draw (568.49,12.36) node [anchor=north west][inner sep=0.75pt]   [align=left] {$\displaystyle v_{2}$};
	\draw (602.49,39.36) node [anchor=north west][inner sep=0.75pt]   [align=left] {$\displaystyle v_{3}$};
	\draw (518.49,93.36) node [anchor=north west][inner sep=0.75pt]   [align=left] {$\displaystyle x_{2}$};
	\draw[decoration={brace,amplitude=5pt}, decorate, line width=1.2pt] (600,115) -- node[below=10pt] {$C_2$} (360,115);
	\draw (62,155) node [anchor=north west][inner sep=0.75pt]   [align=left] {Stage I};
	\draw (155,155) node [anchor=north west][inner sep=0.75pt]   [align=left] {Stage II: the beginning ($C_1$)};
	\draw (410,155) node [anchor=north west][inner sep=0.75pt]   [align=left] {Stage II: the induction};
	
	\draw [fill={rgb, 255:red, 0; green, 0; blue, 0 }  ,fill opacity=1 ]  (51.69, 78.2) circle [x radius= 4, y radius= 4]   ;
	\draw [fill={rgb, 255:red, 0; green, 0; blue, 0 }  ,fill opacity=1 ]  (51.69, 78.2) circle [x radius= 4, y radius= 4]   ;
	\draw [fill={rgb, 255:red, 0; green, 0; blue, 0 }  ,fill opacity=1 ]  (45.56, 35.33) circle [x radius= 4, y radius= 4]   ;
	\draw [fill={rgb, 255:red, 0; green, 0; blue, 0 }  ,fill opacity=1 ]  (45.56, 35.33) circle [x radius= 4, y radius= 4]   ;
	\draw [fill={rgb, 255:red, 0; green, 0; blue, 0 }  ,fill opacity=1 ]  (85.75, 51.46) circle [x radius= 4, y radius= 4]   ;
	\draw [fill={rgb, 255:red, 0; green, 0; blue, 0 }  ,fill opacity=1 ]  (132.85, 113.51) circle [x radius= 4, y radius= 4]   ;
	\draw [fill={rgb, 255:red, 0; green, 0; blue, 0 }  ,fill opacity=1 ]  (121.24, 71.79) circle [x radius= 4, y radius= 4]   ;
	\draw [fill={rgb, 255:red, 0; green, 0; blue, 0 }  ,fill opacity=1 ]  (90.92, 102.7) circle [x radius= 4, y radius= 4]   ;
	\draw [fill={rgb, 255:red, 0; green, 0; blue, 0 }  ,fill opacity=1 ]  (121.24, 71.79) circle [x radius= 4, y radius= 4]   ;
	\draw [fill={rgb, 255:red, 0; green, 0; blue, 0 }  ,fill opacity=1 ]  (272.68, 87.45) circle [x radius= 4, y radius= 4]   ;
	\draw [fill={rgb, 255:red, 0; green, 0; blue, 0 }  ,fill opacity=1 ]  (272.68, 87.45) circle [x radius= 4, y radius= 4]   ;
	\draw [fill={rgb, 255:red, 0; green, 0; blue, 0 }  ,fill opacity=1 ]  (250.94, 50) circle [x radius= 4, y radius= 4]   ;
	\draw [fill={rgb, 255:red, 0; green, 0; blue, 0 }  ,fill opacity=1 ]  (250.94, 50) circle [x radius= 4, y radius= 4]   ;
	\draw [fill={rgb, 255:red, 0; green, 0; blue, 0 }  ,fill opacity=1 ]  (229.38, 87.55) circle [x radius= 4, y radius= 4]   ;
	\draw [fill={rgb, 255:red, 0; green, 0; blue, 0 }  ,fill opacity=1 ]  (229.38, 87.55) circle [x radius= 4, y radius= 4]   ;
	\draw [fill={rgb, 255:red, 0; green, 0; blue, 0 }  ,fill opacity=1 ]  (308, 88) circle [x radius= 4, y radius= 4]   ;
	\draw [fill={rgb, 255:red, 0; green, 0; blue, 0 }  ,fill opacity=1 ]  (308, 88) circle [x radius= 4, y radius= 4]   ;
	\draw [fill={rgb, 255:red, 0; green, 0; blue, 0 }  ,fill opacity=1 ]  (308, 88) circle [x radius= 4, y radius= 4]   ;
	\draw [fill={rgb, 255:red, 0; green, 0; blue, 0 }  ,fill opacity=1 ]  (411.68, 87.45) circle [x radius= 4, y radius= 4]   ;
	\draw [fill={rgb, 255:red, 0; green, 0; blue, 0 }  ,fill opacity=1 ]  (411.68, 87.45) circle [x radius= 4, y radius= 4]   ;
	\draw [fill={rgb, 255:red, 0; green, 0; blue, 0 }  ,fill opacity=1 ]  (368.38, 87.55) circle [x radius= 4, y radius= 4]   ;
	\draw [fill={rgb, 255:red, 0; green, 0; blue, 0 }  ,fill opacity=1 ]  (389.94, 50) circle [x radius= 4, y radius= 4]   ;
	\draw [fill={rgb, 255:red, 0; green, 0; blue, 0 }  ,fill opacity=1 ]  (447, 88) circle [x radius= 4, y radius= 4]   ;
	\draw [fill={rgb, 255:red, 0; green, 0; blue, 0 }  ,fill opacity=1 ]  (447, 88) circle [x radius= 4, y radius= 4]   ;
	\draw [fill={rgb, 255:red, 0; green, 0; blue, 0 }  ,fill opacity=1 ]  (447, 88) circle [x radius= 4, y radius= 4]   ;
	\draw [fill={rgb, 255:red, 0; green, 0; blue, 0 }  ,fill opacity=1 ]  (490.3, 87.89) circle [x radius= 4, y radius= 4]   ;
	\draw [fill={rgb, 255:red, 0; green, 0; blue, 0 }  ,fill opacity=1 ]  (490.3, 87.89) circle [x radius= 4, y radius= 4]   ;
	\draw [fill={rgb, 255:red, 0; green, 0; blue, 0 }  ,fill opacity=1 ]  (447, 88) circle [x radius= 4, y radius= 4]   ;
	\draw [fill={rgb, 255:red, 0; green, 0; blue, 0 }  ,fill opacity=1 ]  (468.56, 50.45) circle [x radius= 4, y radius= 4]   ;
	\draw [fill={rgb, 255:red, 0; green, 0; blue, 0 }  ,fill opacity=1 ]  (468.56, 50.45) circle [x radius= 4, y radius= 4]   ;
	\draw [fill={rgb, 255:red, 0; green, 0; blue, 0 }  ,fill opacity=1 ]  (525.62, 88.45) circle [x radius= 4, y radius= 4]   ;
	\draw [fill={rgb, 255:red, 0; green, 0; blue, 0 }  ,fill opacity=1 ]  (525.62, 88.45) circle [x radius= 4, y radius= 4]   ;
	\draw [fill={rgb, 255:red, 0; green, 0; blue, 0 }  ,fill opacity=1 ]  (447, 88) circle [x radius= 4, y radius= 4]   ;
	\draw [fill={rgb, 255:red, 0; green, 0; blue, 0 }  ,fill opacity=1 ]  (553.49, 69.36) circle [x radius= 4, y radius= 4]   ;
	\draw [fill={rgb, 255:red, 0; green, 0; blue, 0 }  ,fill opacity=1 ]  (553.49, 69.36) circle [x radius= 4, y radius= 4]   ;
	\draw [fill={rgb, 255:red, 0; green, 0; blue, 0 }  ,fill opacity=1 ]  (553.49, 69.36) circle [x radius= 4, y radius= 4]   ;
	\draw [fill={rgb, 255:red, 0; green, 0; blue, 0 }  ,fill opacity=1 ]  (557.63, 26.25) circle [x radius= 4, y radius= 4]   ;
	\draw [fill={rgb, 255:red, 0; green, 0; blue, 0 }  ,fill opacity=1 ]  (553.49, 69.36) circle [x radius= 4, y radius= 4]   ;
	\draw [fill={rgb, 255:red, 0; green, 0; blue, 0 }  ,fill opacity=1 ]  (553.49, 69.36) circle [x radius= 4, y radius= 4]   ;
	\draw [fill={rgb, 255:red, 0; green, 0; blue, 0 }  ,fill opacity=1 ]  (592.89, 51.39) circle [x radius= 4, y radius= 4]   ;
\end{tikzpicture}
			\caption{Constructor's strategy when $S_4$ is forbidden}
			\label{fig : K_3-S_4 : C strat}
		\end{figure}
		
		This strategy is illustrated in Figure \ref{fig : K_3-S_4 : C strat}. We now show that Constructor can indeed play according to this strategy.
		
		\medskip
		
		{\bf Strategy discussion:}
		
		We will start with a claim, which shows that Constructor can create many cherries.
		
		\begin{claim}
			For $n$ large enough, Constructor can follow Stage I.
		\end{claim}
		
		\begin{claimproof}
			During Stage I, write $L$ the set of isolated vertices in $\C$. Let $\ell:=|L|$, and assume $\ell\geq 3\sqrt{n}$. The number of sets $\{v_1,v_2,v_3\}$ with all vertices isolated in $\C$ but somehow connected in $\B$ is at most
			$$
			\sum_{S\in L^3}\sum_{e \in E(\B)}1_{\{e \in S^2\}} \leq \sum_{e \in E(\B)}\ell \leq n\ell < \binom{\ell}{3}
			$$
			for $n$ large enough. Indeed an edge can be in at most $\ell$ sets of three vertices of $L$, and since in Stage I Constructor claims less than $n$ edges, so does Blocker. Thus there exists a set of three vertices of $L$ that are not connected in any way in $\B$.
		\end{claimproof}
		
		\medskip
		
		Now we show that Constructor is able to find an isolated vertex to play the first move of Stage~II afterwards.
		
		\begin{claim}
			For $n$ large enough, Constructor can construct $C_0$.
		\end{claim}
		
		\begin{claimproof}
			At the end of Stage I, there are at least $3\sqrt{n}-3$ vertices that are isolated in $\C$. Since Blocker has played less than $n$ moves during Stage I, the number of vertices that are connected in $\B$ to at least $\sqrt{n}+3$ cherries is at most $\frac{n}{\sqrt{n}+3} < 3\sqrt{n}-3$ (for $n$ large enough). Thus there is a vertex that is isolated in $\C$ and not connected to more than $\sqrt{n}+2$ cherries in $\B$. Let it be $x_0$.
		\end{claimproof}
		
		\medskip
		
		Lastly, let us show that Constructor can build a chain of triangles.
		
		\begin{claim}
			For $n$ large enough, Constructor can follow Stage II.
		\end{claim}
		
		\begin{claimproof}
			By the previous claim, we can construct $C_0$. Now assume that, for some $k$, $C_k$ is constructed (and satisfies the required properties). Assume that there is a cherry that is both connected to less than $\sqrt{n}$ other cherries in $\B$ and that is not connected to $x_k$ in $\B$.
			The construction of $C_{k+1}$ guarantees the first two properties. Since $x_{k+1}$ was part of the selected cherry, it was not connected to more than $\sqrt{n}$ other cherries in $\B$ at that time. Constructor played two edges, thus $x_{k+1}$ is connected to at most $\sqrt{n}+2$ cherries after $C_{k+1}$ is constructed.		
		\end{claimproof}
		
		\medskip
		
		We check that the strategy is valid, meaning that Constructor does not create an $S_4$.
		
		\begin{claim}
			When Constructor follows this strategy, her graph $\C$ remains $S_4$-free.
		\end{claim}
		
		\begin{claimproof}
			A cherry contains no vertex with degree more than $3$, so $\C$ is $S_4$-free at the end of Stage I. When constructing $C_{k+1}$, Constructor first plays an edge between a vertex of degree one ($x_k$) and a vertex of degree 2 ($v_1$), and next an edge between a vertex of degree 2 ($x_k$) and a vertex of degree 1 ($v_2$ or $v_3$). Therefore all vertices remain of degree at most $3$.
		\end{claimproof}
		
		\medskip
		
		Finally, let us show that Constructor created at least $\frac{n}{3}(1+o(1))$ triangles this way.
		
		\begin{claim}
			Constructor can follow Stage II for $k_0 \geq \frac{n}{3}(1+o(1))$ steps.
		\end{claim}
		
		\begin{claimproof}
			We will count the number of vertices that are not in a triangle at the end of the game. There are three possibilities for this to happen: either it was not put into a cherry during Stage I, or it was in a cherry but this cherry was not used by Constructor during Stage II or it is $x_{k_0}$.
			
			Let $\ell_1$ be the number of isolated vertices in $\C$ at the end of Stage I and $\ell_2$ be the number of isolated cherries in $\C$ at the end of Stage II. We know that $\ell_1 < 3\sqrt{n}$. Let us estimate $\ell_2$.
			
			When Constructor ends Stage II, either all of the remaining cherries are connected to at least $\sqrt{n}$ cherries in $\B$, or all of the remaining cherries are connected to $x_{k_0}$ in $\B$. In the first case, Blocker has claimed at least $\frac{\ell_2\sqrt{n}}{2}$ edges between cherries. Since Constructor has claimed less than $\frac23(n-3\sqrt{n}+3) \leq n$ edges during Stage I, and less than $n$ edges during Stage II, we know that $\frac{\ell_2\sqrt{n}}{2} \leq 2n$ and thus $\ell_2 \leq 4\sqrt{n}$. In the second case, since by assumption $x_{k_0}$ is connected to at most $\sqrt{n}+2$ cherries in $\B$, $\ell_2 \leq \sqrt{n}+2$.
			In both cases, $\ell_2 \leq 4\sqrt{n}$, thus the number of vertices in $C_{k_0}$ is at least $n - 3\sqrt{n}-3\cdot4\sqrt{n} - 1 = n - 15\sqrt{n}-1$. Hence $k_0 \geq \frac{n - 15\sqrt{n}-1}{3} = \frac{n}{3}(1+o(1))$.
		\end{claimproof}
		
		\medskip
		
		This claim concludes the proof of Proposition \ref{prop : lower bound g(n,K_3,S_4)}.	
	\end{proof}
	
	\medskip
	
	To obtain the upper bound, we show that Blocker can prevent every vertex from being in several triangles, thus upper bounding the total number of triangles at $\frac{n}{3}$. The details are trickier, since Blocker's strategy has to tackle many cases. We will state this as a new proposition as well:
	
	\begin{prop}\label{prop : upper bound g(n,K_3,S_4)}
		Blocker has a strategy to ensure that the following holds:
		$$
		g(n,K_3,S_4) \leq \frac{n}{3}(1+o(1)) \, .
		$$
	\end{prop}
	
	\begin{proof}[Proof of Proposition~\ref{prop : upper bound g(n,K_3,S_4)}]
		We introduce a criterion that we call \emph{priority}. Blocker will play according to Lemma~\ref{lemma : play in the same cc}, but we will use priority to help Blocker choose the exact edge he should play. In this context, we say that an edge is \emph{free} if it has not been played by any of the two players and if its endvertices are in the same connected component of $\C$ and have degree 1 or 2. Let $e = uv$ be a free edge. Its priority number is:
		\begin{itemize}
			\item[(1)~] if $\C \cup \{e\}$ has 2 more triangles than $\C$.
			\item[(2)~] if (1) is not satisfied and $d_{\C}(u) = d_{\C}(v) = 1$.
			\item[(3)~] if (1) is not satisfied and exactly one vertex among $u,v$ has degree 1 in $\C$.
			\item[(4)~] if (1) is not satisfied and $d_{\C}(u) = d_{\C}(v) = 2$.
		\end{itemize}
		
		
		Assume Constructor plays the edge $e$. If there is a free edge that might complete a triangle with $e$, then Blocker plays it. If there are several such free edges, then he plays the one with the smallest priority number (breaking the ties arbitrarily). If Constructor just completed a triangle $t$ with $e$ and if there is no free edge that might complete another triangle with $e$, then Blocker looks for a free edge that might complete a triangle adjacent to $t$. If there is one, then Blocker blocks it (if there are several such edges, he uses the priority order). Otherwise he plays arbitrarily.
		
		\begin{claim}
			If Blocker follows this strategy, then at any time during the game, at most one edge has priority number (1).
		\end{claim}
		
		\begin{claimproof}
			Since such an edge is immediately claimed by Blocker when it appears, it is enough to verify that two edges of priority (1) cannot appear simultaneously. Edges of priority (1) are diagonals of $C_4$. If an edge is in two $C_4$, then none of the diagonals can be played by Constructor (because at least one of the endvertices has degree 3). Therefore the only way for two edges of priority (1) to appear simultaneously is if Constructor completes an isolated $C_4$ with no diagonal blocked. However, since Blocker has played in the same connected component as Constructor, he necessarily has claimed one of the diagonals before this $C_4$ was completed. Hence the claim is proven.
		\end{claimproof}
		
		\medskip
		
		We now show that each edge in $\C$ cannot be in several triangles. It is enough to show that there is no \emph{diamond} in $\C$ (i.e.~a $C_4$ with one diagonal, see Figure \ref{fig : K_3-S_4 : B strat} for a drawing) whenever Blocker follows the strategy described above.
		
		\begin{claim}
			If Blocker follows this strategy, then no diamond can occur in $\C$.
		\end{claim}
		
		\begin{claimproof}
			Consider the four vertices $v_1,v_2,v_3,v_4$ and the five edges $v_1v_2,v_2v_3,v_3v_4,v_1v_4,v_1v_3$, making a diamond (see Figure \ref{fig : K_3-S_4 : B strat}). We want to show that these edges cannot all be claimed by Constructor. If this were to happen, then vertices $v_2$ and $v_4$ might have degrees 2 or 3. We will consider these different cases, one of which is illustrated in Figure \ref{fig : K_3-S_4 : B strat}.

			\begin{itemize}
				\item {\bf Case 1:} $d_{\C}(v_2) = d_{\C}(v_4) = 2$.
				
				Since the connected component of $v_1,v_2,v_3,v_4$ is of size four, Blocker has always answered in it whenever Constructor has played in it. Therefore, by Lemma \ref{lemma : play in the same cc}, the five edges cannot all be claimed by Constructor during the game.
				
				\item {\bf Case 2:} $d_{\C}(v_2) = 3$, $d_{\C}(v_4) = 2$.
				
				\begin{itemize}
					\item If Constructor manages to play all the edges but $v_1v_3$, then Blocker will block this edge because it will be the only one with priority number (1) (by the first claim).
					\item If Constructor manages to play all the edges but $v_1v_4$, then let $e$ be the last edge played by Constructor among $v_1v_2,v_2v_3,v_3v_4,v_1v_3$.
					\begin{itemize}
						\item If $e = v_1v_3$ or $v_3v_4$, then after this move the only free edge that might complete a triangle with $e$ is $v_1v_4$, so Blocker claims it. This case is illustrated in Figure \ref{fig : K_3-S_4 : B strat}.
						\item If $e = v_1v_2$, then we look back at the time when the last edge between $v_1v_3$ and $v_3v_4$ was played by Constructor. If $v_2$ had degree 2, then $v_1v_4$ was the only edge with priority (2), so Blocker has already claimed it. Otherwise, observing the first case, Blocker would have played two edges among $v_1v_2,v_2v_4,v_1v_4$. Since Constructor played $v_1v_2$ later, Blocker necessarily played $v_1v_4$.
						\item If $e = v_2v_3$, then as before we look back at the time when the last edge between $v_1v_3$ and $v_3v_4$ was played by Constructor. In order to prevent a triangle, Blocker has already played $v_1v_4$ or $v_2v_3$. Since Constructor played $v_2v_3$ later, Blocker necessarily played $v_1v_4$ at that time.
					\end{itemize}
					\item If Constructor manages to play all the edges but $v_1v_2$, then let $e$ be the last edge played by Constructor among $v_2v_3,v_3v_4,v_1v_4,v_1v_3$.
					\begin{itemize}
						\item $e$ cannot be $v_2v_3$ because Constructor cannot build an isolated triangle.
						\item If $e=v_1v_3$, then when Constructor played it Blocker necessarily played $v_1v_2$.
						\item If $e = v_1v_4$, then $v_2$ had degree 1 when all other three edges were claimed (otherwise Blocker would have played $v_1v_4$ because this would have been the only edge with priority (1)). But this is covered by Case 1.
						\item If $e = v_3v_4$, then when the last edge between $v_1v_3$ and $v_1v_4$ was played, Blocker necessarily played $v_1v_2$ or $v_3v_4$ (and it cannot be $v_3v_4$ because this was played by Constructor later).
					\end{itemize}   
				\end{itemize}
				\item {\bf Case 3:} $d_{\C}(v_2) = d_{\C}(v_4) = 3$.
				\begin{itemize}
					\item If Constructor manages to play all the edges but $v_1v_3$, then again Blocker will block it because of the first claim.
					\item Otherwise, by symmetry, assume that Constructor manages to play all the edges but $v_1v_4$ and let $e$ be the last edge played by Constructor among $v_1v_2,v_2v_3,v_3v_4,v_1v_3$.
					\begin{itemize}
						\item If $e = v_1v_3$ or $v_3v_4$, then after this move the only free edge that might complete a triangle with $e$ is $v_1v_4$, thus Blocker claims it. Indeed $v_3$ already has degree 3.
						\item If $e = v_1v_2$, then we look back at the time when the last edge between $v_1v_3, v_2v_3$, and $v_3v_4$ was played by Constructor. If at that time both $v_2$ and $v_4$ had degree one, then at the end of his turn Blocker would have claimed two edges among $v_1v_2,v_1v_4$, and $v_2v_4$. If it is not the case then because of the priority order Blocker claimed $v_1v_2$ or $v_1v_4$, which means that he claimed $v_1v_4$ because we assumed that the other one was played by Constructor later.
						\item If $e = v_2v_3$, then let $v_5$ be the third vertex adjacent to $v_2$ (at the end of the process) and look back at when the last edge between $v_1v_2$ and $v_1v_3$ was played by Constructor. Blocker then either claimed $v_1v_4, v_2v_3,$ or $v_1v_5$. In the latter case, when Constructor plays $v_2v_3$ Blocker will answer according to the part of his strategy when there is no free edge that might complete a triangle with $v_2v_3$ (because $v_3$ will have degree 3), and therefore he will play $v_1v_4$.
					\end{itemize}
				\end{itemize}
			\end{itemize}
			
			\begin{figure}[h!]
				\centering
				\begin{tikzpicture}[x=0.75pt,y=0.75pt,yscale=-1,xscale=1]

\draw    (96,77.33) -- (61,32) ;
\draw    (26,77.33) -- (61,32) ;
\draw    (96,77.33) -- (26,77.33) ;
\draw    (63.74,1.74) .. controls (75.74,16.74) and (48.74,19.74) .. (61,32) ;
\draw    (96,77.33) -- (61,122.67) ;
\draw    (26,77.33) -- (61,122.67) ;
\draw    (63.74,152.92) .. controls (75.74,137.92) and (48.74,134.92) .. (61,122.67) ;
\draw    (273,77.33) -- (238,32) ;
\draw [draw opacity=0]   (203,77.33) -- (238,32) ;
\draw [color={rgb, 255:red, 208; green, 2; blue, 27 }  ,draw opacity=1 ][line width=1.5]    (273,77.33) -- (203,77.33) ;
\draw    (273,77.33) -- (238,122.67) ;
\draw    (203,77.33) -- (238,122.67) ;
\draw    (240.74,152.92) .. controls (252.74,137.92) and (225.74,134.92) .. (238,122.67) ;
\draw    (441,77.33) -- (406,32) ;
\draw [color={rgb, 255:red, 0; green, 0; blue, 0 }  ,draw opacity=1 ][line width=0.75]  [dash pattern={on 0.84pt off 2.51pt}]  (371,77.33) -- (406,32) ;
\draw [color={rgb, 255:red, 208; green, 2; blue, 27 }  ,draw opacity=1 ][line width=1.5]    (441,77.33) -- (371,77.33) ;
\draw    (441,77.33) -- (406,122.67) ;
\draw    (371,77.33) -- (406,122.67) ;
\draw    (408.74,152.92) .. controls (420.74,137.92) and (393.74,134.92) .. (406,122.67) ;
\draw    (297,63) -- (343,63) ;
\draw [shift={(346,63)}, rotate = 180] [fill={rgb, 255:red, 0; green, 0; blue, 0 }  ][line width=0.08]  [draw opacity=0] (8.93,-4.29) -- (0,0) -- (8.93,4.29) -- cycle    ;

\draw (3,73) node [anchor=north west][inner sep=0.75pt]   [align=left] {$\displaystyle v_{1}$};
\draw (71,114) node [anchor=north west][inner sep=0.75pt]   [align=left] {$\displaystyle v_{2}$};
\draw (105,73) node [anchor=north west][inner sep=0.75pt]   [align=left] {$\displaystyle v_{3}$};
\draw (70,17) node [anchor=north west][inner sep=0.75pt]   [align=left] {$\displaystyle v_{4}$};
\draw (180,73) node [anchor=north west][inner sep=0.75pt]   [align=left] {$\displaystyle v_{1}$};
\draw (248,114) node [anchor=north west][inner sep=0.75pt]   [align=left] {$\displaystyle v_{2}$};
\draw (282,73) node [anchor=north west][inner sep=0.75pt]   [align=left] {$\displaystyle v_{3}$};
\draw (247,17) node [anchor=north west][inner sep=0.75pt]   [align=left] {$\displaystyle v_{4}$};
\draw (348,73) node [anchor=north west][inner sep=0.75pt]   [align=left] {$\displaystyle v_{1}$};
\draw (416,114) node [anchor=north west][inner sep=0.75pt]   [align=left] {$\displaystyle v_{2}$};
\draw (450,73) node [anchor=north west][inner sep=0.75pt]   [align=left] {$\displaystyle v_{3}$};
\draw (415,17) node [anchor=north west][inner sep=0.75pt]   [align=left] {$\displaystyle v_{4}$};

\draw [fill={rgb, 255:red, 0; green, 0; blue, 0 }  ,fill opacity=1 ]  (61, 32) circle [x radius= 3, y radius= 3]   ;
\draw [fill={rgb, 255:red, 0; green, 0; blue, 0 }  ,fill opacity=1 ]  (26, 77.33) circle [x radius= 3, y radius= 3]   ;
\draw [fill={rgb, 255:red, 0; green, 0; blue, 0 }  ,fill opacity=1 ]  (26, 77.33) circle [x radius= 3, y radius= 3]   ;
\draw [fill={rgb, 255:red, 0; green, 0; blue, 0 }  ,fill opacity=1 ]  (26, 77.33) circle [x radius= 3, y radius= 3]   ;
\draw [fill={rgb, 255:red, 0; green, 0; blue, 0 }  ,fill opacity=1 ]  (61, 122.67) circle [x radius= 3, y radius= 3]   ;
\draw [fill={rgb, 255:red, 0; green, 0; blue, 0 }  ,fill opacity=1 ]  (61, 122.67) circle [x radius= 3, y radius= 3]   ;
\draw [fill={rgb, 255:red, 0; green, 0; blue, 0 }  ,fill opacity=1 ]  (96, 77.33) circle [x radius= 3, y radius= 3]   ;
\draw [fill={rgb, 255:red, 0; green, 0; blue, 0 }  ,fill opacity=1 ]  (96, 77.33) circle [x radius= 3, y radius= 3]   ;
\draw [fill={rgb, 255:red, 0; green, 0; blue, 0 }  ,fill opacity=1 ]  (238, 32) circle [x radius= 3, y radius= 3]   ;
\draw [fill={rgb, 255:red, 0; green, 0; blue, 0 }  ,fill opacity=1 ]  (203, 77.33) circle [x radius= 3, y radius= 3]   ;
\draw [fill={rgb, 255:red, 0; green, 0; blue, 0 }  ,fill opacity=1 ]  (203, 77.33) circle [x radius= 3, y radius= 3]   ;
\draw [fill={rgb, 255:red, 0; green, 0; blue, 0 }  ,fill opacity=1 ]  (203, 77.33) circle [x radius= 3, y radius= 3]   ;
\draw [fill={rgb, 255:red, 0; green, 0; blue, 0 }  ,fill opacity=1 ]  (238, 122.67) circle [x radius= 3, y radius= 3]   ;
\draw [fill={rgb, 255:red, 0; green, 0; blue, 0 }  ,fill opacity=1 ]  (238, 122.67) circle [x radius= 3, y radius= 3]   ;
\draw [fill={rgb, 255:red, 0; green, 0; blue, 0 }  ,fill opacity=1 ]  (273, 77.33) circle [x radius= 3, y radius= 3]   ;
\draw [fill={rgb, 255:red, 0; green, 0; blue, 0 }  ,fill opacity=1 ]  (273, 77.33) circle [x radius= 3, y radius= 3]   ;
\draw [fill={rgb, 255:red, 0; green, 0; blue, 0 }  ,fill opacity=1 ]  (406, 32) circle [x radius= 3, y radius= 3]   ;
\draw [fill={rgb, 255:red, 0; green, 0; blue, 0 }  ,fill opacity=1 ]  (371, 77.33) circle [x radius= 3, y radius= 3]   ;
\draw [fill={rgb, 255:red, 0; green, 0; blue, 0 }  ,fill opacity=1 ]  (371, 77.33) circle [x radius= 3, y radius= 3]   ;
\draw [fill={rgb, 255:red, 0; green, 0; blue, 0 }  ,fill opacity=1 ]  (371, 77.33) circle [x radius= 3, y radius= 3]   ;
\draw [fill={rgb, 255:red, 0; green, 0; blue, 0 }  ,fill opacity=1 ]  (406, 122.67) circle [x radius= 3, y radius= 3]   ;
\draw [fill={rgb, 255:red, 0; green, 0; blue, 0 }  ,fill opacity=1 ]  (406, 122.67) circle [x radius= 3, y radius= 3]   ;
\draw [fill={rgb, 255:red, 0; green, 0; blue, 0 }  ,fill opacity=1 ]  (441, 77.33) circle [x radius= 3, y radius= 3]   ;
\draw [fill={rgb, 255:red, 0; green, 0; blue, 0 }  ,fill opacity=1 ]  (441, 77.33) circle [x radius= 3, y radius= 3]   ;
\end{tikzpicture}
				\caption{Support for the proof of Proposition \ref{prop : upper bound g(n,K_3,S_4)}: the diamond, and one of the many cases. The curved lines on the first diamond represent edges that might or might not be there. Constructor's last move is in bold red, and Blocker's answer is in dotted line.}
				\label{fig : K_3-S_4 : B strat}
			\end{figure}
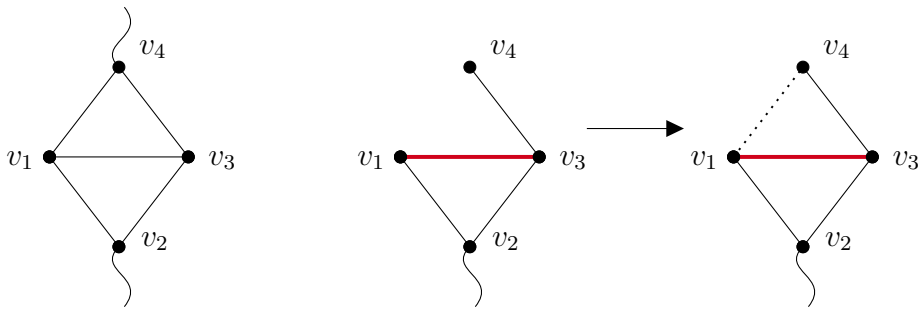
			In all cases, Blocker can prevent a diamond from appearing in $\C$.
		\end{claimproof}
		
		\medskip
		
		Since Constructor's graph will not contain any diamond, each edge in $\C$ can be in at most one triangle. Since each vertex has degree at most $3$ in $\C$, it cannot be in two triangles. This implies that the number of triangles in $\C$ is at most $\frac{n}{3}$.
	\end{proof}
	
	\medskip
	
	\subsubsection{Forbidding $S_5$}
	
	In this section we will prove Theorem~\ref{thm : g(n,K_3,S_5)}. In contrary to the previous case of forbidding $S_4$, here our bounds are not tight. 
	As before we split the proof into two parts and the proofs of two propositions, from which Theorem~\ref{thm : g(n,K_3,S_5)} will then trivially follow. Let us start with the lower bound and a strategy for Constructor.
	
	\medskip
	
	In order to obtain the lower bound, we will reuse the big chain of triangles created via the method presented in the proof of Theorem~\ref{thm : g(n,K_3,S_4)} and show that when we allow Constructor to have vertices of degree $4$, she can roughly double the number of triangles by connecting vertices of this chain in a certain way. We will state the result as a new proposition:
	
	\begin{prop}\label{prop : lower bound g(n,K_3,S_5)}
		Constructor has a strategy to ensure that the following holds:
		$$
		g(n,K_3,S_5) \geq \frac{2n}{3}(1+o(1)) \, .
		$$
	\end{prop}

	\begin{proof}[Proof of Proposition~\ref{prop : lower bound g(n,K_3,S_5)}]
		We provide a strategy for Constructor to build at least $\frac{2n}{3}-o(n)$ triangles in two stages. We illustrate the second stage in Figure~\ref{fig : K_3-S_5 : C strat}.
		
		\medskip
		
		{\bf Stage I:} During this stage, Constructor applies the strategy presented in the proof of Proposition~\ref{prop : lower bound g(n,K_3,S_4)} to obtain a chain of $k_0 = \frac{n}{3}-o(n)$ triangles on $3k_0+1$ vertices. Once she is done constructing this chain, we number these triangles in an arbitrary order. Then, for the $i^{th}$ triangle, let $s_i$ be the vertex of degree 2 in $\C$ and let  $l_i,r_i$ be its neighbours. Afterwards, Constructor proceeds with Stage II.
		
		\medskip
		
		{\bf Stage II:} During this stage, Constructor will pair the triangles created in Stage I and connect those in a certain way. While there exist $i\neq j$ such that the sets $\{s_i,l_i,r_i\}$ and $\{s_j,l_j,r_j\}$ are not connected in $\B$, Constructor will play the following three moves and then repeat this process.
		
		She starts by claiming $s_is_j$. Then pair the edges $s_il_j$ and $s_ir_j$, and also pair the edges $s_jl_i$ and $s_jr_i$.
		If Blocker plays a paired edge, then Constructor plays the other one. Otherwise she plays one of these four edges arbitrarily. Constructor repeats this move once (playing adjacent to $s_j$ if she played adjacent to $s_i$ before). After these two moves, Constructor will have created two more triangles in her graph while maintaining degree at most 4.
		
		Once there are no more free edges as required available any more, Constructor stops playing.

		\begin{figure}[h!]
			\centering
			\begin{tikzpicture}[x=0.75pt,y=0.75pt,yscale=-1,xscale=1]

\draw    (274.68,179.45) -- (310,180) ;
\draw  [color={rgb, 255:red, 0; green, 0; blue, 0 }  ,draw opacity=1 ] (353.3,179.89) -- (310,180) -- (331.56,142.45) -- cycle ;
\draw [draw opacity=0]   (388.62,180.45) -- (310,180) ;
\draw    (353.3,179.89) -- (388.62,180.45) ;
\draw    (331.74,107) -- (331.56,142.45) ;
\draw  [color={rgb, 255:red, 0; green, 0; blue, 0 }  ,draw opacity=1 ] (388.62,70) -- (427.67,51.28) -- (424.35,94.46) -- cycle ;
\draw    (388.62,70) -- (353.3,69.45) ;
\draw  [color={rgb, 255:red, 0; green, 0; blue, 0 }  ,draw opacity=1 ] (310,69.55) -- (353.3,69.45) -- (331.74,107) -- cycle ;
\draw [draw opacity=0]   (274.68,69) -- (353.3,69.45) ;
\draw    (310,69.55) -- (274.68,69) ;
\draw    (427.67,51.28) -- (424.35,94.46) ;
\draw    (310,69.55) -- (331.74,107) ;
\draw [color={rgb, 255:red, 74; green, 144; blue, 226 }  ,draw opacity=1 ][line width=1.5]    (331.56,142.45) .. controls (351.69,128.38) and (366.69,106.38) .. (353.3,69.45) ;
\draw [color={rgb, 255:red, 245; green, 166; blue, 35 }  ,draw opacity=1 ][line width=1.5]    (331.74,107) .. controls (351.88,121.07) and (366.88,143.07) .. (353.48,180) ;
\draw [color={rgb, 255:red, 74; green, 144; blue, 226 }  ,draw opacity=1 ][line width=1.5]    (331.74,142.55) .. controls (311.61,128.49) and (296.61,106.49) .. (310,69.55) ;
\draw  [color={rgb, 255:red, 0; green, 0; blue, 0 }  ,draw opacity=1 ] (274.68,69) -- (235.63,50.28) -- (238.95,93.46) -- cycle ;
\draw    (238.95,93.46) -- (235.63,50.28) ;
\draw  [color={rgb, 255:red, 0; green, 0; blue, 0 }  ,draw opacity=1 ] (388.62,180.45) -- (427.67,199.16) -- (424.35,155.99) -- cycle ;
\draw    (424.35,155.99) -- (427.67,199.16) ;
\draw  [color={rgb, 255:red, 0; green, 0; blue, 0 }  ,draw opacity=1 ] (274.68,179.45) -- (235.63,198.16) -- (238.95,154.99) -- cycle ;
\draw    (235.63,198.16) -- (238.95,154.99) ;
\draw    (427.67,199.16) -- (454.96,227.97) ;
\draw    (235.63,198.16) -- (208.34,226.97) ;
\draw    (427.67,51.28) -- (457.67,23.28) ;
\draw    (235.63,50.28) -- (205.63,22.28) ;
\draw [draw opacity=0]   (457.67,23.28) -- (454.96,227.97) ;
\draw [draw opacity=0]   (205.63,22.28) -- (208.34,226.97) ;
\draw [color={rgb, 255:red, 245; green, 166; blue, 35 }  ,draw opacity=1 ][line width=1.5]    (331.74,107) .. controls (311.61,121.07) and (296.61,143.07) .. (310,180) ;
\draw [draw opacity=0]   (454.96,227.97) -- (208.34,226.97) ;

\draw (339,93) node [anchor=north west][inner sep=0.75pt]   [align=left] {$\displaystyle s_{j}$};
\draw (338,133) node [anchor=north west][inner sep=0.75pt]   [align=left] {$\displaystyle s_{i}$};
\draw (298.34,182.72) node [anchor=north west][inner sep=0.75pt]   [align=left] {$\displaystyle l_{i}$};
\draw (351.31,183.22) node [anchor=north west][inner sep=0.75pt]   [align=left] {$\displaystyle r_{i}$};
\draw (355.34,46.72) node [anchor=north west][inner sep=0.75pt]   [align=left] {$\displaystyle l_{j}$};
\draw (309.34,45.72) node [anchor=north west][inner sep=0.75pt]   [align=left] {$\displaystyle r_{j}$};

\draw [fill={rgb, 255:red, 0; green, 0; blue, 0 }  ,fill opacity=1 ]  (274.68, 179.45) circle [x radius= 4, y radius= 4]   ;
\draw [fill={rgb, 255:red, 0; green, 0; blue, 0 }  ,fill opacity=1 ]  (274.68, 179.45) circle [x radius= 4, y radius= 4]   ;
\draw [fill={rgb, 255:red, 0; green, 0; blue, 0 }  ,fill opacity=1 ]  (310, 180) circle [x radius= 4, y radius= 4]   ;
\draw [fill={rgb, 255:red, 0; green, 0; blue, 0 }  ,fill opacity=1 ]  (353.3, 179.89) circle [x radius= 4, y radius= 4]   ;
\draw [fill={rgb, 255:red, 0; green, 0; blue, 0 }  ,fill opacity=1 ]  (353.3, 179.89) circle [x radius= 4, y radius= 4]   ;
\draw [fill={rgb, 255:red, 0; green, 0; blue, 0 }  ,fill opacity=1 ]  (331.56, 142.45) circle [x radius= 4, y radius= 4]   ;
\draw [fill={rgb, 255:red, 0; green, 0; blue, 0 }  ,fill opacity=1 ]  (331.56, 142.45) circle [x radius= 4, y radius= 4]   ;
\draw [fill={rgb, 255:red, 0; green, 0; blue, 0 }  ,fill opacity=1 ]  (310, 180) circle [x radius= 4, y radius= 4]   ;
\draw [fill={rgb, 255:red, 0; green, 0; blue, 0 }  ,fill opacity=1 ]  (388.62, 180.45) circle [x radius= 4, y radius= 4]   ;
\draw [fill={rgb, 255:red, 0; green, 0; blue, 0 }  ,fill opacity=1 ]  (310, 180) circle [x radius= 4, y radius= 4]   ;
\draw [fill={rgb, 255:red, 0; green, 0; blue, 0 }  ,fill opacity=1 ]  (388.62, 70) circle [x radius= 4, y radius= 4]   ;
\draw [fill={rgb, 255:red, 0; green, 0; blue, 0 }  ,fill opacity=1 ]  (388.62, 70) circle [x radius= 4, y radius= 4]   ;
\draw [fill={rgb, 255:red, 0; green, 0; blue, 0 }  ,fill opacity=1 ]  (427.67, 51.28) circle [x radius= 4, y radius= 4]   ;
\draw [fill={rgb, 255:red, 0; green, 0; blue, 0 }  ,fill opacity=1 ]  (427.67, 51.28) circle [x radius= 4, y radius= 4]   ;
\draw [fill={rgb, 255:red, 0; green, 0; blue, 0 }  ,fill opacity=1 ]  (424.35, 94.46) circle [x radius= 4, y radius= 4]   ;
\draw [fill={rgb, 255:red, 0; green, 0; blue, 0 }  ,fill opacity=1 ]  (427.67, 51.28) circle [x radius= 4, y radius= 4]   ;
\draw [fill={rgb, 255:red, 0; green, 0; blue, 0 }  ,fill opacity=1 ]  (427.67, 51.28) circle [x radius= 4, y radius= 4]   ;
\draw [fill={rgb, 255:red, 0; green, 0; blue, 0 }  ,fill opacity=1 ]  (353.3, 69.45) circle [x radius= 4, y radius= 4]   ;
\draw [fill={rgb, 255:red, 0; green, 0; blue, 0 }  ,fill opacity=1 ]  (353.3, 69.45) circle [x radius= 4, y radius= 4]   ;
\draw [fill={rgb, 255:red, 0; green, 0; blue, 0 }  ,fill opacity=1 ]  (331.74, 107) circle [x radius= 4, y radius= 4]   ;
\draw [fill={rgb, 255:red, 0; green, 0; blue, 0 }  ,fill opacity=1 ]  (331.74, 107) circle [x radius= 4, y radius= 4]   ;
\draw [fill={rgb, 255:red, 0; green, 0; blue, 0 }  ,fill opacity=1 ]  (353.3, 69.45) circle [x radius= 4, y radius= 4]   ;
\draw [fill={rgb, 255:red, 0; green, 0; blue, 0 }  ,fill opacity=1 ]  (274.68, 69) circle [x radius= 4, y radius= 4]   ;
\draw [fill={rgb, 255:red, 0; green, 0; blue, 0 }  ,fill opacity=1 ]  (238.95, 93.46) circle [x radius= 4, y radius= 4]   ;
\draw [fill={rgb, 255:red, 0; green, 0; blue, 0 }  ,fill opacity=1 ]  (235.63, 50.28) circle [x radius= 4, y radius= 4]   ;
\draw [fill={rgb, 255:red, 0; green, 0; blue, 0 }  ,fill opacity=1 ]  (235.63, 50.28) circle [x radius= 4, y radius= 4]   ;
\draw [fill={rgb, 255:red, 0; green, 0; blue, 0 }  ,fill opacity=1 ]  (388.62, 180.45) circle [x radius= 4, y radius= 4]   ;
\draw [fill={rgb, 255:red, 0; green, 0; blue, 0 }  ,fill opacity=1 ]  (427.67, 199.16) circle [x radius= 4, y radius= 4]   ;
\draw [fill={rgb, 255:red, 0; green, 0; blue, 0 }  ,fill opacity=1 ]  (424.35, 155.99) circle [x radius= 4, y radius= 4]   ;
\draw [fill={rgb, 255:red, 0; green, 0; blue, 0 }  ,fill opacity=1 ]  (424.35, 155.99) circle [x radius= 4, y radius= 4]   ;
\draw [fill={rgb, 255:red, 0; green, 0; blue, 0 }  ,fill opacity=1 ]  (238.92, 155.31) circle [x radius= 4, y radius= 4]   ;
\draw [fill={rgb, 255:red, 0; green, 0; blue, 0 }  ,fill opacity=1 ]  (238.95, 154.99) circle [x radius= 4, y radius= 4]   ;
\draw [fill={rgb, 255:red, 0; green, 0; blue, 0 }  ,fill opacity=1 ]  (235.63, 198.16) circle [x radius= 4, y radius= 4]   ;
\draw [fill={rgb, 255:red, 0; green, 0; blue, 0 }  ,fill opacity=1 ]  (235.63, 198.16) circle [x radius= 4, y radius= 4]   ;
\draw [fill={rgb, 255:red, 0; green, 0; blue, 0 }  ,fill opacity=1 ]  (454.96, 227.97) circle [x radius= 4, y radius= 4]   ;
\draw [fill={rgb, 255:red, 0; green, 0; blue, 0 }  ,fill opacity=1 ]  (457.67, 23.28) circle [x radius= 4, y radius= 4]   ;
\draw [fill={rgb, 255:red, 0; green, 0; blue, 0 }  ,fill opacity=1 ]  (454.96, 227.97) circle [x radius= 4, y radius= 4]   ;
\draw [fill={rgb, 255:red, 0; green, 0; blue, 0 }  ,fill opacity=1 ]  (205.63, 22.28) circle [x radius= 4, y radius= 4]   ;
\draw [fill={rgb, 255:red, 0; green, 0; blue, 0 }  ,fill opacity=1 ]  (310, 69.55) circle [x radius= 4, y radius= 4]   ;
\draw [fill={rgb, 255:red, 0; green, 0; blue, 0 }  ,fill opacity=1 ]  (208.34, 226.97) circle [x radius= 4, y radius= 4]   ;
\end{tikzpicture}
			\caption{The second phase of Constructor's strategy in the game $K_3-S_5$. The blue edges are paired together, and so are the orange ones.}
			\label{fig : K_3-S_5 : C strat}
		\end{figure}
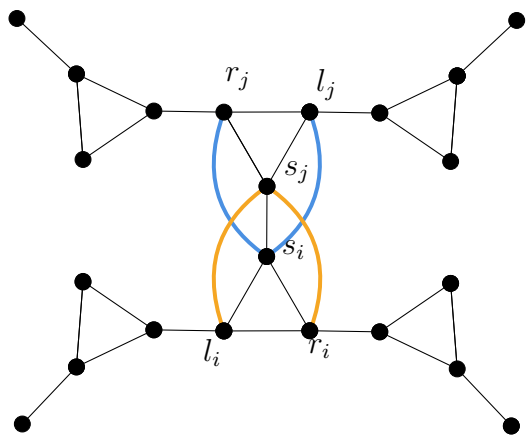

		\medskip
		
		{\bf Strategy discussion:}
		
		The proof of Proposition \ref{prop : lower bound g(n,K_3,S_4)} ensures that Constructor can play according to this strategy for Stage I, and there is nothing to check for Stage II. We now show that this strategy gives the right number of triangles.
		
		\begin{claim}
			During Stage II, all but at most $o(n)$ triangles are paired.
		\end{claim}
		
		\begin{claimproof}
			Denote by $r$ the number of triangles that are in the chain $C_{k_0}$ but that are not paired to any other triangle during Stage II. At the end of the game, all of these $r$ triangles are pairwise connected in $\B$.
			During the whole game, Constructor has played less than $2n$ edges, which means, that Blocker also played less than $2n$ edges. Therefore we get $\frac{r(r-1)}{2} < 2n$, and solving this in $r$ yields $r \leq 2\sqrt{n}$ (for $n$ large enough).
		\end{claimproof}
		
		\medskip
		
		All the triangles that were paired added one more triangle in $\C$ (actually each pair of two triangles added two more triangles). Since the number of triangles at the end of Stage I was $\frac{n}{3}-o(n)$, the number of triangles created by Constructor during Stage II is $\left(\frac{n}{3}-o(n)\right)-o(n)$. Finally, the number of triangles in $\C$ at the end of the game is $\frac{2n}{3}-o(n)$.	
	\end{proof}
	
	\medskip
	
	Now let us turn our attention to the upper bound and a strategy for Blocker. To obtain the upper bound of Theorem~\ref{thm : g(n,K_3,S_5)}, we show that Blocker can prevent any $K_4$ from appearing in Constructor's graph. By looking at all possibilities for a vertex to be in many (i.e.~more than three) triangles, we show that the maximum number of triangles is bounded by $n$. We also state this result as a new proposition:
	
	
	\begin{prop}\label{prop : upper bound g(n,K_3,S_5)}
		Blocker has a strategy to ensure that the following holds:
		$$
		g(n,K_3,S_5) \leq n(1+o(1)) \, .
		$$
	\end{prop}
	
	
	\begin{proof}[Proof of Proposition~\ref{prop : upper bound g(n,K_3,S_5)}.]
		We now investigate the upper bound. We first provide a strategy for Blocker ensuring that Constructor does not build a $K_4$, and we will deduce that via this strategy Blocker prevents Constructor from building more than $n$ triangles. Let us state and prove the following lemma:
		
		\begin{lem}\label{lemma : g(n,K_4,S_5)}
			For all $n\geq 0$, Blocker has a strategy such that the following holds:
			$$
			g(n,K_4,S_5)=0 \, .
			$$
		\end{lem}
		
		\begin{proof}
			The strategy we provide for Blocker is very similar to the one presented in the proof of Proposition~\ref{prop : upper bound g(n,K_3,S_4)}, except that there are more cases to handle. Whenever possible, Blocker will play according to Lemma~\ref{lemma : play in the same cc} and adjacent to Constructor's last move (with some additional priority rules), otherwise Blocker will play arbitrarily.
			
			Similarly as in the proof of Proposition \ref{prop : upper bound g(n,K_3,S_4)}, an edge is said to be \emph{free} if it has not been played by any of the two players and if its endvertices are in the same connected component of $\C$ and have degree 1, 2, or 3. Let $e = uv$ be a free edge. Its priority number is:
			\begin{itemize}
				\item[(0)~] if $\C \cup \{e\}$ has three more triangles than $\C$, and additionally there is no edge between the remaining vertices of these triangles.
				\item[(1)~] if (0) is not satisfied but if $\C \cup \{e\}$ has three more triangles than $\C$.
				\item[(2)~] if (0) and (1) are not satisfied and if $\C \cup \{e\}$ has two more triangles than $\C$.
				\item[(3)~] if (0), (1), and (2) are not satisfied and $\C \cup \{e\}$ has one more edge that is in three triangles than $\C$.
				\item[(4)~] if (0), (1), (2), and (3) are not satisfied and we have $d_{\C}(u) \leq 2$ as well as $d_{\C}(u) \leq 2$.
				\item[(5)~] otherwise.
			\end{itemize}
			
			Assume Constructor plays the edge $e$. If there is a free edge that might complete a triangle containing $e$, then Blocker plays it. If there are several such free edges, then he plays the one with the smallest priority number (breaking ties arbitrarily). Otherwise Blocker plays arbitrarily.
			
			\medskip
			
			The only way for Constructor to create a $K_4$ is to create a \emph{dangerous diamond} that will not be blocked by Blocker. We call a dangerous diamond a $C_4$ in $\C$ with one diagonal in $\C$ and the other diagonal being free. Since edges of priority (0) or (1) are diagonals of dangerous diamonds, Blocker claims such edges as soon as they appears. Assume that at some point during the game all dangerous diamonds have been blocked (thus so far there is no $K_4$). Let us show that two dangerous diamonds cannot occur in one Constructor's move. This will prove the lemma, because Blocker will be able to block all dangerous diamonds immediately when they appear and thus prevent any $K_4$.
			
			\medskip
			
			There are two possibilities for two dangerous diamonds to appear at the same time: either these diamonds share exactly one common side edge, or they share a common diagonal. We present these two possibilities in Figure \ref{fig : dangerous diamonds 1}, the black edges have to be claimed by Constructor, the yellow ones need to be free for the two diamonds to be dangerous. We also write $\emptyset$ next to vertices that have degree less than $4$ in $\C$ but cannot have degree $4$ in order for the diamonds to be dangerous.
			
			\begin{figure}[h!]
				\centering
				\input{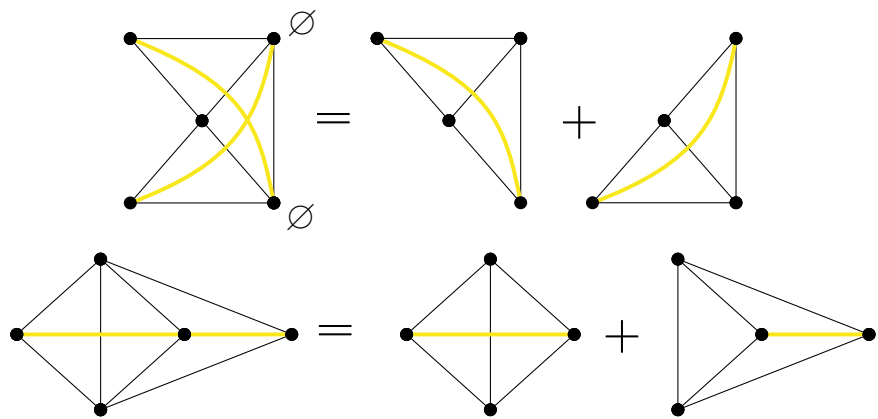}
				\caption{The two possibilities for having two dangerous diamonds simultaneously. In black, Constructor's edges. In yellow, edges that need to be free.}
				\label{fig : dangerous diamonds 1}
			\end{figure}
			
			We examine these two cases and show that Constructor cannot have all the black edges while the yellow ones remain free. To to this, we will further distinguish between black edges whether they have been played by Constructor during her last turn (red) or at a previous point in time (blue). An illustration can be found in Figure~\ref{fig : dangerous diamonds 2}.
			
			\begin{figure}[h!]
				\centering
				\input{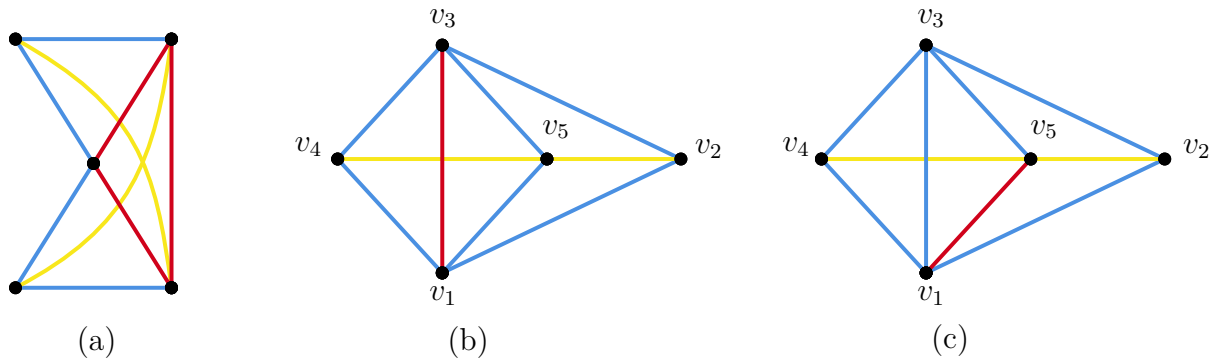}
				\caption{Support for the proof of Lemma \ref{lemma : g(n,K_4,S_5)}. In blue and red, Constructor's edges, the last one being red. In yellow, edges that need to be free when Constructor plays her last edge.}
				\label{fig : dangerous diamonds 2}
			\end{figure}
			
			Now we present an overview of the cases that can happen during the course of the game and how Blocker's strategy handles all these cases.
			
			\begin{itemize}
				\item {\bf Case 1:} There is exactly one common side edge.
				
				This configuration is presented in Figure~\ref{fig : dangerous diamonds 2}.a. The last edge claimed by Constructor needs to be a red one, so that both dangerous diamonds appear at the same time, and the two yellow diagonals need to be free. When the second edge among the red ones was added, Blocker necessarily claimed a yellow, blue, or red edge since there was no other adjacent free edge.
				
				\item {\bf Case 2:} There is a common diagonal.
				
				The two dangerous diamonds are labelled $v_1v_4v_3v_5$ and $v_1v_2v_3v_5$, the diagonals that need to be free being $v_4v_5$ and $v_2v_5$ (in yellow in Figures~\ref{fig : dangerous diamonds 2}b and~\ref{fig : dangerous diamonds 2}c).
				\begin{itemize}
					\item If Constructor manages to play every edge but the common diagonal $v_1v_3$ (Figure~\ref{fig : dangerous diamonds 2}.b), then this edge has priority (0) or (1). If it has priority (0), then it is the only one (note that if two edges of priority (0) would appear at the same time, it would mean that Blocker would not have always played adjacent to Constructor's moves when he could have), hence Blocker plays it and Constructor cannot claim this red edge. 
					
					If this edge has priority (1), then either $v_2v_4,v_2v_5$, or $v_4v_5$ was already blocked. There can be at most one other edge with priority (0) or (1) in this connected component. Since $v_2,v_4,v_5$ can be connected to at most one other vertex outside of $v_1,v_2,v_3,v_4,v_5$, the edge with priority (0) or (1) can only be one of the following:
					\begin{itemize}
						\item $v_4v_5$: This corresponds to $v_4$ and $v_5$ sharing a third common neighbour. But if Blocker claims this one, then not all of the yellow edges will be free.
						\item $v_2v_4$: In this case we said that an edge among $v_2v_5$ and $v_4v_5$ (a yellow edge) was already blocked.
					\end{itemize}
					In each case, Constructor cannot obtain a graph as in Figure~\ref{fig : dangerous diamonds 2}.b, either because she will not be able to claim the red edge, or because some yellow edge will have already been claimed by Blocker.
					\item If Constructor manages to play every edge but say $v_1v_5$ (Figure~\ref{fig : dangerous diamonds 2}.c), then $v_1$ and $v_3$ cannot have any other adjacent edges in $\C$, while $v_2,v_4,v_5$ can each have at most one. Note that $v_1v_3$ has to be the first edge claimed by Constructor among all of the black edges in Figure~\ref{fig : dangerous diamonds 1}, otherwise Blocker would have blocked one of the others. Let $e$ be the last edge played by Constructor among $v_1v_2,v_1v_4,v_2v_3,v_3v_4,v_3v_5$. Now, we distinguish between the following cases for her choice of $e$:
					\begin{itemize}
						\item $v_1v_4$: When Constructor played the last edge among $v_3v_4$ and $v_3v_5$, then clearly no edge of priority at most (3) can have been created, thus Blocker blocked either one of the edges $v_1v_4,v_1v_5,v_2v_5$, or $v_4v_5$ (which are yellow, red, or blue edges) or Blocker blocked $v_2v_4$. In the latter case, when Constructor played $e$, Blocker blocked $v_1v_5$ because there was no free edge of priority (0), (1), or (2), and this was the only one of priority (3) (blocking it would prevent $v_1v_3$ from being in three triangles). So either Constructor cannot claim $v_1v_5$ (red), or a yellow edge is blocked.
						\item $v_3v_4$: Without loss of generality, assume $v_1$ reached degree 3 in $\C$ before $v_3$ did. At that time, Blocker played an edge of priority (4), thus she claimed $v_2v_3$ or $v_3v_4$ (which is not possible because we know for sure that Constructor claimed those later), or $v_2v_4$. In the latter case, when Constructor played $e$, Blocker blocked $v_4v_5$ because there was no free edge of priority (0), (1), (2), or (3), and this was the only one of priority (4).
						\item $v_3v_5$: If this was the last edge Constructor claimed, then just before that the vertices $v_1,v_2,v_3,v_4$ formed a dangerous diamond, and we assumed that at that time Blocker could block all dangerous diamonds, so that is what he did by claiming $v_1v_5$.
						\item We handle all the remaining cases by symmetry.
					\end{itemize}
				\end{itemize}
			\end{itemize}
			
			After examining all cases, we showed that against Blocker's strategy, Constructor cannot create two dangerous diamonds simultaneously. Thus during the whole game every dangerous diamond is neutralized by Blocker and therefore no $K_4$ appears in $\C$.
		\end{proof}
		
		To prove Proposition~\ref{prop : upper bound g(n,K_3,S_5)}, we will apply Blocker's strategy described in Lemma~\ref{lemma : g(n,K_4,S_5)}, and show that when there does not appear a $K_4$ in $\C$ the number of triangles Constructor claims is at most $n$ (which means that every vertex will be on average in at most 3 triangles).
		We argue that there cannot be many vertices which are in more than 3 triangles, more precisely that for every such vertex there are some vertices that are in less than 3 triangles.
		
		Naturally, we first investigate the possibilities for a vertex to be in four triangles, in a graph that has maximal degree at most 4 and no $K_4$.
		The only way for a vertex $v$ to be in four triangles is to be the center of a wheel of size 4, that is, to have a cycle $v_1v_2v_3v_4$ in $\C$ such that $vv_1,\dots,vv_4 \in \C$.
		Since the degree of every $v_i$ in $\C$ is at most 4, it can only be in a triangle if this triangle contains at least another vertex among $v,v_1,v_2,v_3,v_4$. By symmetry, five cases might occur and are illustrated in Figure~\ref{fig : W_4}:
		
		\begin{itemize}
			\item[A:] No vertex $v_i$ is in another triangle.
			\item[B:] There is exactly one additional triangle with some other vertex ($v_1v_2u_1$).
			\item[C:] There are exactly two additional triangles with two other vertices ($v_1v_2u_1$ and $v_3v_4u_2$).
			\item[D:] There are exactly two additional triangles with one other vertex ($v_1v_2u_1$ and $v_1v_4u_1$).
			\item[E:] There are exactly three additional triangles with one other vertex ($v_1v_2u_1$, $v_1v_4u_1$, and $v_2v_3u_1$). We can immediately exclude this structure for the following reason. If this structure were to happen, then, since all vertices are of degree 4, this would be a complete connected component. Constructor would have twelve edges in a connected component of six vertices, which is in contradiction to Lemma~\ref{lemma : play in the same cc}.
		\end{itemize}
		
		\begin{figure}[h!]
			\centering
			\input{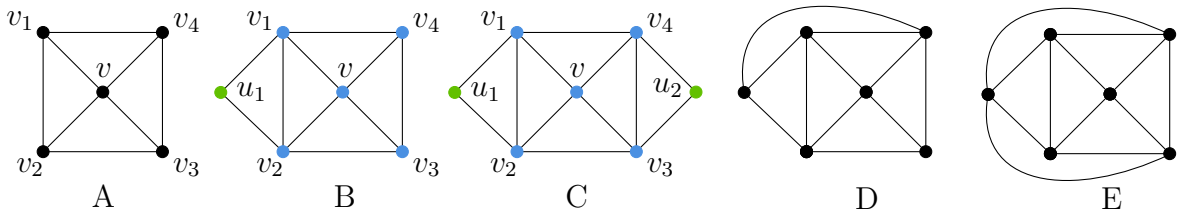}
			\caption{The structures of types A,B,C,D,E that appear in the proof of Proposition \ref{prop : upper bound g(n,K_3,S_5)}}
			\label{fig : W_4}
		\end{figure}
		
		We will denote by $S_\text{A},\dots,S_\text{D}$ the sets of all structures of type A,$\dots$,D in Constructor's graph at the end of the game. Note that the vertices that are in a structure of type A or D cannot be in another structure. When we consider structures of type B or C, the vertices $v,v_1,\dots,v_4$ are in a unique such structure, while the vertices $u_1$ and $u_2$ are in at most two structures of such types. Therefore, we can partition the set of all vertices into the following disjoint sets:
		\begin{itemize}
			\item $V_\text{A}$: the set of vertices that are in a structure of type A.
			\item $V_\text{D}$: the set of vertices that are in a structure of type D.
			\item $V_\alpha$: the set of vertices that are in a structure of type B or C, and labeled as $v,v_1,v_2,v_3$, or $v_4$ (colored in blue in Figure \ref{fig : W_4}).
			\item $V_\beta$: the set of vertices that are in at least one structure of type B or C, labeled as $u_1$ or $u_2$ (colored in green in Figure \ref{fig : W_4}).
			\item $V_\text{U}$: the vertices that are in no structure of type A,B,C, or D.
		\end{itemize}
		
		For any vertex $v$, denote by $t(v)$ the number of triangles containing $v$. The total number of triangles in Constructor's graph is 
		$$
		\frac{1}{3}\sum_{v\in V} t(v) = \frac{1}{3}\left(\sum_{v\in V_\text{A}}t(v) + \sum_{v\in V_\text{D}}t(v) + \sum_{v\in V_\alpha}t(v) + \sum_{v\in V_\beta}t(v) + \sum_{v\in V_\text{U}}t(v)\right) \, .
		$$
		Let us upper bound these sums.
		
		\begin{itemize}
			\item By assumption, all the vertices that are in four triangles are in a structure of type A,B,C, or D, thus $\sum_{v\in V_\text{U}}t(v) \leq 3|V_\text{U}|$.
			\item The vertices that are in a structure of type A or D are in a unique such structure, thus we can partition the sets $V_\text{A}$ and $V_\text{D}$ according to the structures the vertices belong to. Moreover, these vertices cannot be in a triangle with vertices that are not in this structure (simply by assumption on $v,v_1,\dots,v_4$ and because $u_1$ is already of degree 3).
			Therefore,
			$$
			\sum_{v\in V_\text{A}}t(v) = \sum_{s \in S_\text{A}}\sum_{v\in s}t(v) = \sum_{s\in S_\text{A}} 12 = \frac{12}{5}|V_\text{A}|
			$$
			and
			$$
			\sum_{v\in V_\text{D}}t(v) = \sum_{s \in S_\text{D}}\sum_{v\in s}t(v) = \sum_{s\in S_\text{D}} 18 = \frac{18}{6}|V_\text{D}| \, .
			$$
			\item The vertices in $V_\beta$ that are in a structure $s$ can be in at most one triangle outside of $s$, either in another structure of type B or C, or with some other vertices of the graph.		
			Therefore, 
			$$
			\sum_{v\in V_\alpha\cup V_\beta}t(v) \leq \sum_{s \in S_\text{B} \cup S_\text{C}}\sum_{v \in s\cap V_\alpha}t(v) + 2|V_\beta| = 14|S_\text{B}| + 16|S_\text{C}| + 2|V_\beta| \, .
			$$		
			Since the vertices in $V_\alpha$ are in a unique structure, $|V_\alpha| = 5|S_\text{B}|+5|S_\text{C}|$. Moreover, since every structure of type C contains two vertices of $V_\beta$ and since a vertex of $V_\beta$ is in at most two structures, it holds that $|S_\text{C}|\leq |V_\beta|$.	Therefore, 
			$$
			\sum_{v\in V_\alpha\cup V_\beta}t(v) \leq 3|V_\alpha| + |V_\beta| + 2|V_\beta| \, .
			$$
		\end{itemize}
		
		Putting it all together, we get that the following:
		$$
		\frac13 \sum_{v \in V}t(v) \leq \frac{1}{3}\left(\frac{12}{5}|V_\text{A}| + \frac{18}{6}|V_\text{D}| + 3|V_\alpha| + 3|V_\beta| + 3|V_\text{U}|\right) \leq \frac13 \cdot3|V| = n \, .
		$$
		Therefore, the number of triangles in Constructor's graph is at most $n$.	
	\end{proof}

	\medskip
	
	\subsubsection{Forbidding larger stars.}
	
	The strategies for Blocker for $k=4$ and $k=5$ are extremely tricky because we basically explore all of the possible cases. We were not able to find a more general, nicer approach, not even for these small $k$. However, we present an idea for a lower bound when $k$ goes to infinity but does not grow too fast in comparison to $n$.
	
	The idea behind this result is that Constructor can create many stars of size $k(n)$ and then maximizes the number of triangles between the leaves of each star, playing independent copies of the game where we do not forbid anything (see Section~\ref{subsec : res K_3-empty}). We illustrate this in Figure \ref{fig : K_3-S_k : C strat}. Even though Blocker can claim some edges while Constructor creates the stars, the number of triangles Blocker blocks becomes negligible.
	
	This bound does not seem tight to us, because there might be many triangles between different stars that we have not accounted for in the lower bound. Moreover, it does not give us any information about what happens when $k$ is fixed.
	
	\medskip
	
	\begin{figure}[h!]
		\centering
		\input{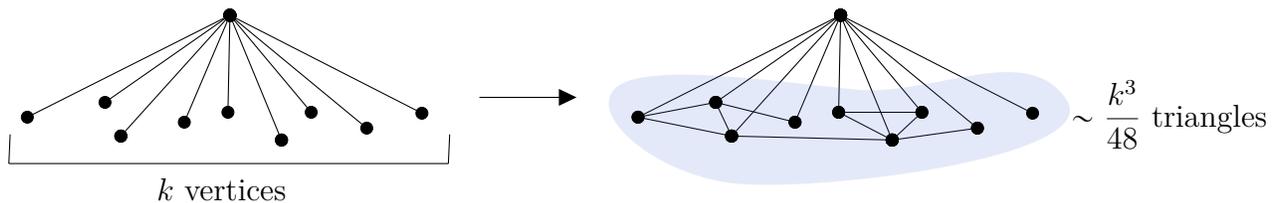}
		\caption{Constructor's strategy for Theorem \ref{thm : g(n,K_3,S_k)} (one of the stars)}
		\label{fig : K_3-S_k : C strat}
	\end{figure}
	
	\begin{proof}[Proof of Theorem \ref{thm : g(n,K_3,S_k)}]
		We will provide a strategy for Constructor. She will start by creating a a star of size $k$. Then, whenever Blocker plays between leaves of this star, so does Constructor, according to Erd\H{o}s-Selfridge's strategy. Whenever it is not the case and while there are at least $3k\sqrt{n}$ isolated vertices in $\C$, Constructor picks a set of $2k+1$ such vertices that are also isolated in $\B$ and creates another star of size $k$.
		
		We could also look at her strategy as playing on many stars by themselves, where every star that is created by Constructor becomes a game independent from what happens in the rest of the graph, since Constructor always answers in a star whenever Blocker claims an edge between its leaves.
		
		Let us state the following claim, which tells us that Constructor can find large sets of independent vertices for a long time.
		\begin{claim}
			While there are at least $3k\sqrt{n}$ isolated vertices in $\C$, Constructor can find a set of $2k+1$ vertices that are isolated in $\C$ and not connected in $\B$. 
		\end{claim}
		\begin{claimproof}
			Similarly as in previous proofs, let $L$ be the set of isolated vertices in $\C$ and $\ell := |L|$. Assume $l \geq 3k\sqrt{n}$.
			Then the number of sets of size $2k+1$ of isolated vertices in $\C$, with at least one Blocker edge, is at most 
			$$
			\sum_{S \in L^{2k+1}}\sum_{e \in E(\B)} 1_{\{|S \cap e|=2 \}} \ \leq \sum_{e \in E(\B) \text{ not inside of a star}}\binom{\ell}{2k-1} \, .
			$$
			Since Constructor will use less than $n$ moves to construct her stars, the number of Blocker edges that are not between leaves of a star is also less than $n$. Thus the previous sum is less than
			$$
			n \binom{\ell}{2k-1} < \binom{\ell}{2k+1}
			$$
			for $n$ large enough.
		\end{claimproof}
		
		\medskip
		
		Inside of every star, Constructor tries to maximize the number of triangles between the leaves of this star. Once Constructor finishes a new star, Blocker might already have claimed (at most) $k$ edges among the leaves, that would appear in $k(k-1)$ triangles. Thanks to Theorem~\ref{thm : ES}, we know that Blocker can play in such a way, that Constructor does not claim more than $\frac{k^3}{48}(1+o(1)) - O(k^2)$ triangles per star. Therefore, the total number of triangles Constructor creates is at least $\frac{n-3k\sqrt{n}}{k+1}\left(\frac{k^3}{48}- O(k^2)\right)$. Since $k(n) = o(\sqrt{n})$, this becomes $\frac{nk^2}{48}(1+o(1))$.
	\end{proof}
	
	\medskip

	We can also wonder what happens when $k$ is greater than $\sqrt{n}$.
	Constructor could secretly pick $k+1$ vertices and play the game where nothing is forbidden only in the smaller induced subgraph $K_{k+1}$. Thanks to Theorem \ref{thm : g(n,K_3,emptyset)}, Constructor can create at least $\frac{k^3}{48}$ triangles.
	In particular, if $k$ is linear in $n$, then this tells us that $g(n,K_3,S_{k+1})$ is of order $n^3$. More precisely, if $k=cn$ with $c \in [0,1]$, then $g(n,K_3,S_{k+1}) \geq \frac{c^3n^3}{48}(1+o(1))$. However, we do not know the order of $g(n,K_3,S_{k+1})$ when $k$ is of order between $\sqrt{n}$ and $n$, especially we wonder if it is of the same order as $ex(n,K_3,S_{k+1})$ -- that is $nk^2$. Another interesting question is when $k = \frac{n}{2}$. We observed via some simulations that if both players follow the Erd\H{o}s-Selfridge's strategy that gave the result when there is no constraint (see Theorem \ref{thm : g(n,K_3,emptyset)} and Theorem \ref{thm : ES}), then Constructor's graph had maximal degree $\frac{n}{2}$. However we do not know how Blocker would play with the knowledge that Constructor cannot add any edge to a vertex of high degree. We wonder if the result is the same as with no degree constraint 
	or if the constant factor is different.

	\section{Results on the planar and embedded Constructor-Blocker game}\label{sec : results planarity}

	In this section, we will focus on our versions of game where we focus on planarity and present our proofs of Theorem~\ref{thm : PCB} and Theorem~\ref{thm : ECB}.
	
	\subsection{The Planar Constructor-Blocker game}\label{subsec : pcb}
	
	We will start with the planar Constructor-Blocker game and the proof of Theorem~\ref{thm : PCB}. 
	Before we proceed to prove our result, let us first state a trivial upper bound to $g_{\text{PCB}}(n,K_3)$.
	
	\begin{prop}[\cite{HS79}]
		For $n\geq 3$, the maximum number of triangles in a planar graph is $3n-8$.
	\end{prop}
	
	The extremal graphs matching this number are called \emph{Apollonian networks} (see Figure \ref{fig : Apollonian}) and can be constructed inductively. Start with one triangle, then in each step add a new vertex to some face and connect it to the three vertices of this face.
	
	\begin{figure}[h!]
		\centering
		\begin{tikzpicture}[x=0.75pt,y=0.75pt,yscale=-1,xscale=1]

\draw   (95.73,86.44) -- (13.12,86.6) -- (54.28,16) -- cycle ;
\draw   (203.23,86.44) -- (120.62,86.6) -- (161.78,16) -- cycle ;
\draw    (120.62,86.61) -- (161.88,63.01) ;
\draw    (161.88,63.01) -- (203.23,86.44) ;
\draw    (161.88,63.01) -- (161.78,16) ;
\draw   (312,86.44) -- (229.39,86.6) -- (270.55,16) -- cycle ;
\draw    (229.39,86.61) -- (270.65,63.01) ;
\draw    (270.65,63.01) -- (312,86.44) ;
\draw    (270.65,63.01) -- (270.55,16) ;
\draw    (270.55,16) -- (283.44,55.41) ;
\draw    (283.44,55.41) -- (312,86.44) ;
\draw    (283.44,55.41) -- (270.65,63.01) ;
\draw    (13.12,86.61) -- (54.28,16) ;
\draw    (95.73,86.44) -- (54.28,16) ;

\draw [fill={rgb, 255:red, 0; green, 0; blue, 0 }  ,fill opacity=1 ]  (229.4, 86.6) circle [x radius= 3, y radius= 3]   ;
\draw [fill={rgb, 255:red, 0; green, 0; blue, 0 }  ,fill opacity=1 ]  (312, 86.43) circle [x radius= 3, y radius= 3]   ;
\draw [fill={rgb, 255:red, 0; green, 0; blue, 0 }  ,fill opacity=1 ]  (270.55, 16) circle [x radius= 3, y radius= 3]   ;
\draw [fill={rgb, 255:red, 0; green, 0; blue, 0 }  ,fill opacity=1 ]  (270.55, 16) circle [x radius= 3, y radius= 3]   ;
\draw [fill={rgb, 255:red, 0; green, 0; blue, 0 }  ,fill opacity=1 ]  (312, 86.43) circle [x radius= 3, y radius= 3]   ;
\draw [fill={rgb, 255:red, 0; green, 0; blue, 0 }  ,fill opacity=1 ]  (270.65, 63.01) circle [x radius= 3, y radius= 3]   ;
\draw [fill={rgb, 255:red, 0; green, 0; blue, 0 }  ,fill opacity=1 ]  (270.65, 63.01) circle [x radius= 3, y radius= 3]   ;
\draw [fill={rgb, 255:red, 0; green, 0; blue, 0 }  ,fill opacity=1 ]  (270.65, 63.01) circle [x radius= 3, y radius= 3]   ;
\draw [fill={rgb, 255:red, 0; green, 0; blue, 0 }  ,fill opacity=1 ]  (283.44, 55.41) circle [x radius= 3, y radius= 3]   ;
\draw [fill={rgb, 255:red, 0; green, 0; blue, 0 }  ,fill opacity=1 ]  (283.44, 55.41) circle [x radius= 3, y radius= 3]   ;
\draw [fill={rgb, 255:red, 0; green, 0; blue, 0 }  ,fill opacity=1 ]  (120.63, 86.6) circle [x radius= 3, y radius= 3]   ;
\draw [fill={rgb, 255:red, 0; green, 0; blue, 0 }  ,fill opacity=1 ]  (203.23, 86.43) circle [x radius= 3, y radius= 3]   ;
\draw [fill={rgb, 255:red, 0; green, 0; blue, 0 }  ,fill opacity=1 ]  (161.78, 16) circle [x radius= 3, y radius= 3]   ;
\draw [fill={rgb, 255:red, 0; green, 0; blue, 0 }  ,fill opacity=1 ]  (161.88, 63.01) circle [x radius= 3, y radius= 3]   ;
\draw [fill={rgb, 255:red, 0; green, 0; blue, 0 }  ,fill opacity=1 ]  (161.88, 63.01) circle [x radius= 3, y radius= 3]   ;
\draw [fill={rgb, 255:red, 0; green, 0; blue, 0 }  ,fill opacity=1 ]  (13.13, 86.6) circle [x radius= 3, y radius= 3]   ;
\draw [fill={rgb, 255:red, 0; green, 0; blue, 0 }  ,fill opacity=1 ]  (54.28, 16) circle [x radius= 3, y radius= 3]   ;
\draw [fill={rgb, 255:red, 0; green, 0; blue, 0 }  ,fill opacity=1 ]  (54.28, 16) circle [x radius= 3, y radius= 3]   ;
\draw [fill={rgb, 255:red, 0; green, 0; blue, 0 }  ,fill opacity=1 ]  (95.73, 86.44) circle [x radius= 3, y radius= 3]   ;
\end{tikzpicture}
		\caption{The Apollonian networks on three, four, and five vertices}
		\label{fig : Apollonian}
	\end{figure}
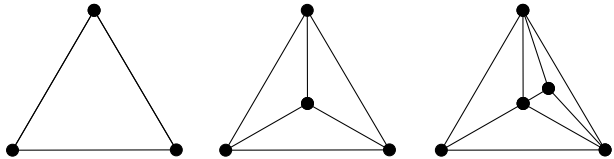
	
	Since Constructor can match this bound from below, we will only consider a strategy for her. 
	Before we go into more details, let us briefly give a sketch of Constructor's strategy. She will try to build a graph which is very close to an Apollonian network. At the beginning of the game, Constructor will take a few turns to build a fan of size 6. Then, as in the inductive construction of the Apollonian networks, Constructor will be able to create three triangles with every new vertex: She will connect one new vertex to the center of the fan, and then (since she does not need to tell Blocker on which face this vertex will end up) she has many possibilities to create new triangles, and Blocker cannot block all of them.
	


	\begin{proof}[Proof of Theorem \ref{thm : PCB}]	
		In the following, we will denote the center of an $F_k$ by $v$ and the $k$ other vertices that form a path by $v_1,\dots,v_k$. We call a vertex red if it is the center of a wheel of size at least 4 or of a fan of size at least 6. We call a vertex orange if this is not true, but it is the center of a $W_3$. Orange vertices might turn red during the game, and the red vertices will remain red forever.
		Before we state Constructor's strategy, we will prove two useful claims.
		
		\begin{claim}\label{clm : building F_6}
			Constructor can build an $F_6$ using 12 moves and 8 vertices.
		\end{claim}
		
		\begin{claimproof}
			Following the strategy of Lemma~\ref{lemma : pendant leg}, Constructor can build a triangle with a pendant leg within four moves (note that this does not break any planarity constraints), and afterwards she arbitrarily denotes these vertices as $v_1,v_2$, and $v_3$. Then (after Blocker's turn) she picks a vertex $u$ that is isolated from this triangle, both in $\C$ and $\B$, and claims the edge $uv_1$. Now, Blocker cannot block both $uv_2$ and $uv_3$, so Constructor can claim one of those edges after Blocker's next turn. Constructor successively adds four vertices this way and creates a fan with center $v_1$ and leaves (after relabelling them) $w_1,\dots,w_6$ such that $w_1w_2,\dots,w_5w_6 \in \C$. She can create such a fan within a total of $4+ 4\cdot2 = 12$ moves.
		\end{claimproof}
		
		\medskip
		
		We will say that a vertex $u$ is \emph{safe} from a $W_4$ (resp. a $W_5$ or an $F_6$) if it is not connected to any of the five (resp. six or seven) vertices of this structure in $\B$. The following claim will be crucial for Constructor in the creation of three new triangles for every vertex she adds to her graph.
		
		\begin{claim}{\textit{\textbf{(Crucial claim)}}}
			If $u$ is safe from a $W_4$, a $W_5$, or an $F_6$, then, in three moves, Constructor can build three new triangles with $v$, and $v$ becomes an orange vertex.
		\end{claim}
		
		\begin{claimproof}
			The following strategy is illustrated in Figure \ref{fig : PCB}.
			\begin{itemize}
				\item {\bf Case 1:}  If $v$ is the center of a $W_4$ (with leaves $v_1,\dots,v_4$) and $u$ is not connected to $v,v_1,\dots,v_4$ in $\B$, then Constructor plays $uv$. By symmetry, assume that Blocker plays $uv_4$ (or Blocker does not connect any of the $v_i$ to $u$). Constructor then plays $uv_2$. Now, Blocker cannot block both $uv_1$ and $uv_3$, thus Constructor can build a $K_4$ centered at $u$.
				This strategy can easily be adapted to the case of a $W_5$.
				\item {\bf Case 1:}  If $v$ is the center of an $F_6$ (with leaves $v_1,\dots,v_6$) and $u$ is not connected to $v,v_1,\dots,v_6$ in $\B$, then Constructor plays $uv$. By symmetry, assume that Blocker plays $uv_i$ for some $i\geq 4$ (or that Blocker does not connect any of the $v_i$ to $u$). Constructor then plays $uv_2$. Now, Blocker cannot block both $uv_1$ and $uv_3$, thus Constructor can build a $K_4$ centered at $u$. Furthermore $u$ becomes the center of a $W_3$ (with leaves among $v,v_1,\dots,v_6$).
			\end{itemize}
		\end{claimproof}
		
		\medskip
		
		\begin{figure}[h!]
			\centering
			\input{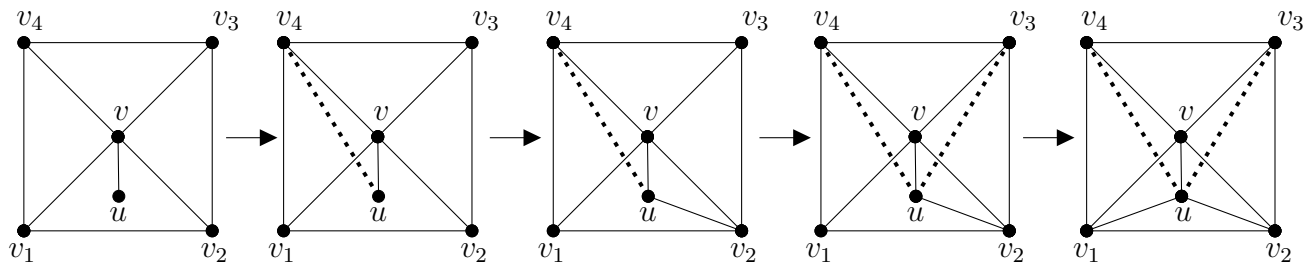}
			\vspace{-13pt}
			\caption{Constructor's strategy to create three triangles with a new vertex $u$ and a $W_4$. Constructor's moves are in plain lines, Blocker's moves are in dotted lines.}
			\label{fig : PCB}
		\end{figure}
		
		With these claims at hand, we can now describe the global strategy for Constructor in three stages.
		
		\medskip
		
		{\bf Stage I:} Constructor creates an $F_6$ using at most twelve moves and eight vertices, then she proceeds to Stage II.
		
		\medskip
		
		{\bf Stage II:} During this stage, Constructor builds up a set $\mathcal{S}$ of structures (for example fans of size 6) that will used later to create further triangles. She repeats the following process $\frac{n}{105}$ times:
		\begin{enumerate}
			\item First she picks a red vertex $v$ that is not the center of a structure in $\mathcal{S}$. Let $S \in \mathcal{S}$ be an associated $W_4,W_5$, or $F_6$.
			
			\emph{Remark.} The first vertex she picks will be the center of the $F_6$ created in Stage I. Afterwards, new red vertices will appear during this stage.
			\item Constructor applies successively the crucial claim 26 times, with vertices that are isolated in $\C$ and safe from $S$. By the pigeonhole principle, there will be at least one face of $S$ that contains at least six of these vertices. These vertices, together with $v$, will form a new $F_6$, which we then add to $\mathcal{S}$. Note that these six vertices were in no structure of $\mathcal{S}$ before. They will be either orange or red, but at least one of them will be red and can take over the role of $v$ during the next iteration of Stage II.
		\end{enumerate}
		Once Constructor is done with Stage II, she proceeds to Stage III.
		
		\medskip
		
		{\bf Stage III:} While there are more than $6\cdot 105 +1$ isolated vertices in $\C$, Constructor picks an isolated vertex $u$ and a structure $S \in \mathcal{S}$ such that $u$ is safe from $S$. She then creates three triangles according to the strategy in the crucial claim. When there are not enough isolated vertices available any more, she stops playing.
		
		\medskip
		
		{\bf Strategy discussion:}
		
		Let us show, that Constructor can actually follow her strategy and during this process creates at least $3n - o(n)$ triangles.
		She can follow Stage I because of Claim~\ref{clm : building F_6}. For Stage II, we will show that there are enough vertices available for Constructor as desired.
		
		The total number of edges that Constructor (and therefore also Blocker) plays during Stage I and Stage II combined is upper bounded by $12 + 3\cdot26\cdot\frac{n}{105} = 12 + \frac{78n}{105}$. Thus, if we pick a fixed $S \in \mathcal{S}$, at most $12 + \frac{78n}{105}$ isolated vertices are not safe from $S$, and we can find enough safe vertices to repeat the process during Stage II.
		
		For Stage III, note that a vertex can be in at most two elements of $\mathcal{S}$. Indeed, it can be used only once as a center of a fan, and only once as a leaf of a fan (because the leaves of fans are vertices that were not in an element of $\mathcal{S}$ before). Thus one Blocker edge can make one vertex not safe from at most two elements of $\mathcal{S}$. During the whole game, Constructor -- and therefore Blocker as well -- will play less than $3n$ edges.
		Therefore, at any time during Stage III, the number of pairs $(u,S)$ which are unavailable for Constructor is upper bounded by
		\begin{align*}
			|\{(u,S) \,|\, u \text{ isolated in }\C, S \in \mathcal{S}, u\text{ not safe from } S\}| & = \sum_{S \in \mathcal{S}}\sum_{\ u \text{ isolated in $\C$}}\sum_{\ uv \in E(\B)} 1_{\{v \in S\}}\\
			& = \sum_{uv \in E(\B)}\sum_{\ S \in \mathcal{S}} 1_{\{u \text{ isolated in }\C\}} 1_{\{v \in S\}}\\
			& \leq 2 \cdot |E(\B)|\\
			& \leq 6n \, .
		\end{align*}
		
		However, during Stage III, the following holds true:
		$$
		|\{(u,S) \,|\, u \text{ isolated in }\C, S \in \mathcal{S}\}| \geq \frac{n}{105}\cdot (6\times 105+1) = 6n + \frac{n}{105} > 6n \, .
		$$	
		Thus, Constructor can always find a pair $(u,S)$ with $u$ isolated in $\C$ and safe from $S$, and construct three new triangles with $u$.
		
		Since Constructor can follow her strategy, at the end of the game she used $n-o(n)$ vertices which each created 3 triangles, thus the number of triangles is asymptotically equivalent to $3n$.	
	\end{proof}
	
	\medskip
	
	\subsection{The Embedded Constructor-Blocker game}
	
	This section is devoted to the proof of Theorem \ref{thm : ECB}. Since here our bounds do not match, we split the proof into two parts, corresponding to the lower bound (Proposition~\ref{prop : lower bound g_{ECB}}) and the upper bound (Proposition~\ref{prop : upper bound g_{ECB}}). As before, we first present the main ideas for both player's strategies to achieve their respective bounds, as well as a result for the non-game version to compare our results to.
	
	\medskip
	
	To get this aforementioned comparison for our result, let us briefly look at the maximum number of triangles a planar graph embedded with vertex set $\mathbb{U}_n$ (the $n^{\text{th}}$ roots of unity) can have. Such a graph is an outerplanar graph, that is a planar graph where all vertices are on the same face, and this question reduces to the maximum number of triangles in an outerplanar graph with $n$ vertices.
	%
	For such graphs it is known that these can have at most $n-2$ triangles.
	
	\medskip
	
	Now let us talk about strategies. Constructor will play a strategy in two phases. First, she claims edges between vertices that are (almost) consecutive and thus she creates a polygon of size between $\frac{n}{2}$ and $n$. Then she builds many triangles inside of this polygon, basically by taking edges between vertices that are close to each other.
	
	\medskip
	
	Considering a strategy for Blocker, he can claim around $\frac{n}{3}$ distinct outer edges (i.e.~edges between two consecutive vertices). At the end of the game, we can contract these edges (i.e.~merge their endvertices) without losing any triangle of $\C$. With this operation we end up with a graph on at most $\frac{2n}{3}$ vertices, thus the number of triangles is at most $\frac{2n}{3}-2$.
	
	\medskip
	
	In both strategies, the outer edges seem crucial. Our intuition is the following: Whenever Constructor plays some edge $e$, all the edges crossing $e$ become useless since Constructor cannot play them any more. Thus Blocker does not have to block all these edges and can thus focus on somewhere else. The further away the endvertices of the edge $e$ are, the more edges it crosses and thus the more it helps Blocker in his job of blocking many triangles. In the same vein, Constructor wants the endvertices to be close to each other to not destroy too many potential triangles.
	
	\medskip
	
	Now, let us present a strategy for Constructor for the lower bound. Therefore, let us state Constructor's side as a new proposition:
	
	\begin{prop}\label{prop : lower bound g_{ECB}}
		Constructor has a strategy to ensure that the following holds:
		$$
		g_{\mathrm{ECB}}(n,K_3) \geq (1+o(1))\frac{n}{2} \, .
		$$
	\end{prop}
	
	\begin{proof}[Proof of Proposition~\ref{prop : lower bound g_{ECB}}]
		We provide a strategy for Constructor in two stages:
		
		\medskip
		
		{\bf Stage I:} During the first $m$ turns of the game, Constructor creates a polygon $\P$ of size $m$ (where $\frac{n}2 \leq m \leq n$). Details of how she achieves this can be found in the strategy discussion. Afterwards, Constructor proceeds to Stage II.
		
		\medskip
		
		{\bf Stage II:}  Constructor creates a graph that is close to a triangulation of $\P$, i.e. she creates many triangles inside of $\P$. Details can be found in the strategy discussion. Once there are no more edges available, Constructor stops playing.
		
		\medskip
		
		{\bf Strategy discussion:}
		
		We first focus on Stage I, when Constructor wants to build a large polygon with not so many Blocker edges inside of it. Let us state and prove the following claim:
		
		\begin{claim}
			For any $n\geq 1$, Constructor can claim all (maybe except for one) edges of a polygon $\P$, that has $m \in \left[ \frac{n}{2},n \right]$ edges, and such that Blocker will have claimed at most $2m-n$ edges inside this polygon.
		\end{claim}
		
		\begin{claimproof}
			We give the precise strategy for Constructor to achieve this goal. First, we order the vertices $v_1,\dots,v_n$ clockwise. Then Constructor starts by playing $v_1v_2$, and then initialises the sets $V_C := \{v_1,v_2\}$ and $V_B := \emptyset$.
			
			Now, whenever Blocker plays $v_iv_j$ with $i<j$, if $v_j \notin V_B \cup V_C$, then we add it to $V_B$, otherwise we do not change the sets.
			
			Next, assume it is Constructor's turn and let $i:= \max\{j,v_j \in V_C\}$ and $k := \min\{j>i,v_j \notin V_B\}$. Constructor then plays $v_iv_k$ and we add $v_k$ to $V_C$. If at some point $\{j>i,v_j \notin V_B\} = \emptyset$, then Constructor plays $v_iv_1$ (if this edge was not taken by Blocker). For Stage II, we will consider that this edge is claimed by Blocker, if it is not the case then the number of triangles completed by Constructor might decrease by one. The polygon $\P$ is made of the vertices of $V_C$ that are connected in increasing order. An example is given in Figure \ref{fig : ECB : polygon}.
			
			\begin{figure}[h!]
				\centering
				\input{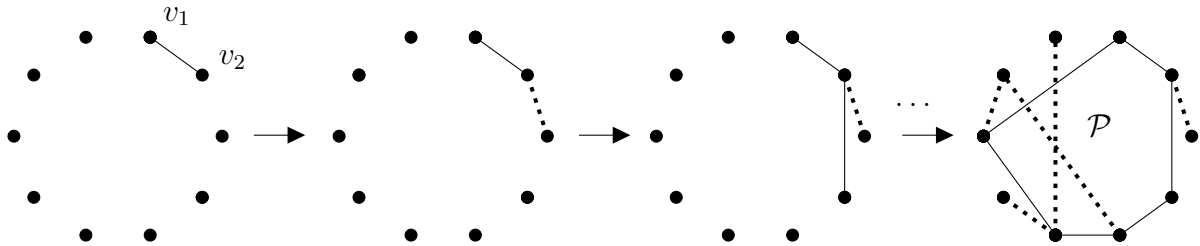}
				\caption{An example of Constructor creating her polygon, in the early stages of the process and at the end. Blocker's edges are displayed as dotted lines.}
				\label{fig : ECB : polygon}
			\end{figure}
			
			Now let us analyze the number of Blocker edges inside of $\P$ with regards to the size of $\P$.
			Denote by $m$ the number of edges -- and therefore of vertices -- in $\P$. Since $V(C) \cup V(B) = \mathbb{U}_n$ and $|V_B| \leq |V_C|$, we immediately have $m \geq \frac{n}{2}$. Moreover, the number of Blocker edges between vertices of $V_C$ is at most $m-(n-m) = 2m-n$. Indeed each player played $m$ edges and from these Blocker used $n-m$ moves to put all of the $n-m$ vertices into $V_B$ (thus he played edges that are outside of $\P$).
		\end{claimproof}
		
		\medskip
		
		We now explain how Constructor can (partially) triangulate $\P$ during Stage II.
		We forget about the vertices in $V_B$ and Blocker's edges that were adjacent to those vertices. We relabel the vertices of our polygon $\P_0 := \P$ as $v_1(0) < v_2(0)<\dots<v_m(0)$. Once both players will have played $t$ further moves, we will consider another polygon $\P_t$ (see below) of size $s_t$ and denote its vertices by $v_1(t)<\dots<v_{s_t}(t)$. We will also denote by $l_t$ the number of Blocker's edges between vertices of $\P_t$, and by $T_t$ the number of triangles in $\C$ at this time (when each player has played $m+t$ moves in total). Recall that $l_0 \leq 2m-n$.
		
		\medskip
		
		We will call a \emph{$k$-ear} of $\P_t$ an edge between $v_i(t)$ and $v_{i+k+1}(t)$ for some $i$. The strategy for Constructor when she has to play her $(t+1)^{th}$ move is the following:
		
		\medskip
		
		Assume that there is at least one Blocker edge and one available edge in $\P_t$. Let $k_t = k$ be the minimal integer such that all of the $1$-ears, $2$-ears,\dots, $(k-1)$-ears of $\P_t$ are blocked but not all of the $k$-ears. Two cases arise and will be illustrated in Figure \ref{fig : ECB : tiangulation}:
		\begin{itemize}
			\item[(a)] If there is a Blocker's edge in $\P_t$ that is not a $1$-ear,\dots,($k-1$)-ear, then there is an available $k$-ear that crosses such an edge, which Constructor can claim as her next move. Then she keeps playing in $\P_{t+1}$, the polygon with $s_{t+1}=s_t-k$ vertices that she just created.
			\item[(b)] Otherwise, Constructor plays any ($k+1$)-ear, say $v_1(t)v_{k+2}(t)$, and now consider the polygon $\P_{t+1}$ with $s_{t+1}=s_t-k-1$ vertices that she just created. Whenever Blocker plays inside of $\P_{t+1}$, Constructor answers according to her strategy, and whenever Blocker claims $v_1(t)v_{k+1}(t)$ or $v_2v_{k+2}(t)$ Constructor claims the other one of those edges, thus creating one triangle. We can therefore assume that the latter immediately happens without loss of generality.
		\end{itemize}
		
		\begin{figure}[h!]
			\centering
			\input{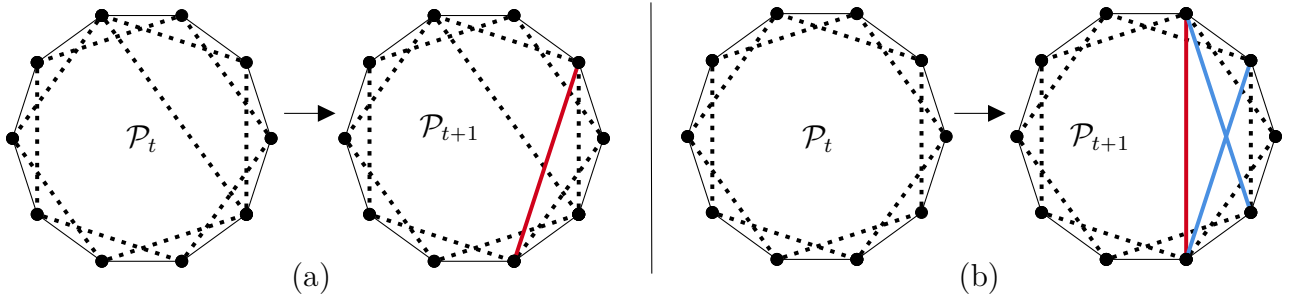}
			\vspace{-11pt}
			\caption{Constructor's strategy to triangulate her polygon. In both examples, $k=2$, Blocker's edges are depicted in dotted lines, and the red edge is Constructor's next move. In case (b), Constructor takes a ($k+1$)-ear, and pairs the blue edges, thus she will be able to create a triangle with one of them.}
			\label{fig : ECB : tiangulation}
		\end{figure}
		
		Constructor stops this strategy at $t_f$ when either all of the edges inside $\P_t$ are claimed by Blocker or all of the edges inside $\P_t$ are free. In the latter case, the following claim shows that Constructor then has a strategy to perfectly triangulate $\P_t$.

		\begin{claim}\label{clm : triangulate P_t}	
			If at time $t_f$ there is no Blocker edge in $\P_{t_f}$, then Constructor can perfectly triangulate $\P_{t_f}$ and thus create $s_{t_f}-2$ more triangles. 
		\end{claim}
		
		\begin{claimproof}
			We proceed by induction on $s_{t_f}$. If $s_{t_f}=4$, then, even if there is one Blocker edge, Constructor is sure to completely triangulate her polygon. If $s_{t_f}\geq 5$, then Constructor can claim any $1$-ear and build one triangle. Blocker will be able to play only one edge in $\P_{t+1}$ and Constructor will then claim a 1-ear crossing this edge, and so on.
		\end{claimproof}
		
		\medskip

		Let us compute the number of triangles $T_{t_f}$ next.
		
		\begin{claim}
			During the game, if Constructor follows her strategy, the number of triangles in her graph is at least
			$$
			T_{t_f} \geq m-s_{t_f} - \frac{l_0-l_{t_f}}{2} \, .
			$$
		\end{claim}
		
		\begin{claimproof}		
			Let us see how $T_t,s_t$, and $l_t$ evolve depending on the cases:
			\begin{itemize}
				\item[(a)] Let us say that the ear Constructor plays is $v_i(t)v_{i+k+1}(t)$. Then for all $1 \leq j \leq k-1$, the $k+j+1$ edges $v_{i-j}(t)v_{i+1}(t),\dots,v_{i+k}(t)v_{i+k+j+1}(t)$ plus the additional existing Blocker edge were in $\P_t$ but not in $\P_{t+1}$. In one move Blocker can claim one more edge in $\P_{t+1}$. Thus $l_t-l_{t+1} \geq \sum_{j=1}^k(k+j-1)+1-1 = \frac{(k-1)(3k+2)}{2}$. Therefore, in this case we get the following:
				$$
				\begin{matrix*}[l]
					T_{t+1} = T_t + 1_{\{k=1\}} \, ,\\
					s_{t+1} = s_t - k \, ,\\
					l_{t+1} \leq l_t - \frac{(k-1)(3k+2)}{2} \, .
				\end{matrix*}
				$$
				\item[(b)] Note that since in this situation there is at least one Blocker edge and $k\geq 2$. For all $1 \leq j \leq k-1$, the $k+j+2$ edges $v_{s_t-j+1}(t)v_{2}(t),\dots,v_{k+1}(t)v_{k+2+j}(t)$ were in $\P_t$ but not in $\P_{t+1}$. In one move Blocker can claim one more edge in $\P_{t+1}$. Therefore, in this case we get the following:
				$$
				\begin{matrix*}[l]
					T_{t+1} = T_t + 1 \, ,\\
					s_{t+1} = s_t - k-1 \, ,\\
					l_{t+1} \leq l_t - \frac{(k-1)(3k+4)}{2}+1 \leq l_t - \frac{(k-1)(3k+2)}{2} \, .
				\end{matrix*}
				$$
			\end{itemize}
			
			Thus, using telescopic sums we get the following:
			\begin{align*}
				l_0-l_{t_f} & = \sum_{t=0}^{t_f-1} \left(l_t-l_{t+1}\right) \\
				& \geq \sum_{t= 0}^{t_f-1} \frac{(k_t-1)(3k_t+2)}{2} = \sum_{t= 0}^{t_f-1} 1_{\{k_t\geq 2\}}\frac{(k_t-1)(3k_t+2)}{2} \\
				& \geq \sum_{t= 0}^{t_f-1} 1_{\{k_t\geq 2\}}\frac{4k_t}{2}
			\end{align*}
			and
			\begin{align*}
				m-s_{t_f}-T_{t_f} & = \sum_{t=0}^{t_f-1} \left(s_t-s_{t+1}+T_t-T_{t+1}\right) = \sum_{t=0}^{t_f-1} (k_t - 1_{\{k_t=1\}}) = \sum_{t=0}^{t_f-1} k_t1_{\{k_t\geq 2\}} \\
				& \leq \frac{l_0-l_{t_f}}{2} \, ,
			\end{align*}
			thus we have 
			$$
			T_{t_f} \geq m-s_{t_f} - \frac{l_0-l_{t_f}}{2} \, .
			$$
		\end{claimproof}
		
		\medskip
		
		Finally let us bound the number of triangles in $\C$ at the end of the game.
		
		\begin{itemize}
			\item If at time $t_f$ all the edges inside of $\P_{t_f}$ are claimed by Blocker, then
			$$
			l_{t_f} = \frac{s_{t_f}(s_{t_f}-1)}{2}-s_{t_f} = \frac{s_{t_f}(s_{t_f}-3)}{2}
			$$
			and
			$$
			T_{t_f} \geq m - \frac{l_0}{2} + s_{t_f}\cdot \frac{s_{t_f}-7}{4} \geq m-4-\frac{l_0}{2} \, .
			$$
			\item Otherwise, all the edges inside of $\P_{t_f}$ are free, and we know from Claim~\ref{clm : triangulate P_t} that Constructor can perfectly triangulate $\P_{t_f}$.
			Therefore, in this case, the number of triangles in $\C$ at the end of the game is at least $T_{t_f} + s_{t_f}-2 \geq m-\frac{l_0}{2}$.
		\end{itemize}

		In both cases, using $l_0 \leq 2m-n$, we get that the number of triangles in $\C$ at the end of the game is at least $\frac{n}{2}-4$ (actually $\frac{n}{2}-5$ if we consider the case where the polygon $\P$ was not actually closed).
	\end{proof}
	
	Let us switch to Blocker's side now. The upper bound we provide does not match our lower bound, but still gives a non-trivial upper bound on the number of triangles Constructor can create. We can state this as a new proposition:
	
	\begin{prop}\label{prop : upper bound g_{ECB}}
		Blocker has a strategy to ensure that the following holds:
		$$
		g_{\mathrm{ECB}}(n,K_3) \leq (1+o(1))\frac{2n}{3} \, .
		$$
	\end{prop}
	
	\begin{proof}[Proof of Proposition~\ref{prop : upper bound g_{ECB}}]
		The strategy we present for Blocker consists of taking as many non-adjacent outer edges as possible. 
		We first count the number of such edges Blocker can claim.
		
		\begin{claim}
			Blocker can claim at least $\lfloor\frac{n}{3}\rfloor$ non-adjacent outer edges.
		\end{claim}
		
		\begin{claimproof}
			We give a strategy for Blocker.
			Assume that Constructor's first move is an outer edge (or else can pretend that she claimed an outer edge). Order the outer edges $\{1,\dots,n\}$ clockwise, the one claimed by Constructor being the $n^{th}$ one. Blocker then claims the $(n-1)^{th}$ edge. Now group the edges from 1 to $3\cdot \lfloor\frac{n-2}{3}\rfloor$ in groups of three, and assign a type to each edge, clockwise $a,b,c,a,b,c,\dots$
			\begin{itemize}
				\item Whenever Constructor plays an edge of type $a$ (resp. $b$), Blocker plays the adjacent edge of type $b$ (resp. $a$).
				\item If Constructor plays an edge of type $c$ or an inner edge, then Blocker plays an arbitrary edge of type $a$ or $b$ that is not adjacent to an edge he has already claimed. If at some point he would claim a particular edge of type $a$ or $b$ that he has already claimed, he plays an arbitrary edge of type $a$ or $b$ as well.
				\item When Blocker cannot play according to this strategy, he plays arbitrarily.
			\end{itemize}
			
			An example of this strategy is given in Figure \ref{ECB : B strat}. 	
			At the end of the game, Blocker will have an edge of type $a$ or $b$ in all blocks of three consecutive edges, thus he will have claimed $\lfloor\frac{n-2}{3}\rfloor + 1 \geq \lfloor\frac{n}{3}\rfloor$ non-adjacent outer edges.
		\end{claimproof}
		
		\medskip
		
		The next claim will now give us an upper bound on the number of triangles Constructor can still create when many non-adjacent outer edges are claimed by Blocker.
		
		\begin{claim}
			If Blocker has claimed $m$ non-adjacent outer edges, then Constructor can create at most $n-m-2$ triangles.
		\end{claim}
		
		\begin{claimproof}
			Let $e_1,\dots,e_m$ be the non-adjacent outer edges claimed by Blocker. The idea is to contract these edges to obtain a bound on the number of triangles in $\C$. We consider Constructor's graph $\C$ at the end of the game. Let $\hat{\C}$ be the graph obtained from $\C$ by contracting every edge $e_i$.
			
			Every triangle in $\C$ corresponds to a triangle in $\hat{\C}$, because it cannot have an edge $e_i$ as an edge. Thus the number of triangles in $\C$ is upper bounded by the number of triangles in $\hat{\C}$, which is at most $|V(\hat{\C})|-2$, hence $n-m-2$. Note that we use that the edges are not adjacent. Indeed if two edges $v_1v_2$ and $v_2v_3$ were adjacent, then contracting these edges would mean contracting $v_1v_3$ as well, but this edge might be claimed by $\C$ and inside a triangle.
		\end{claimproof}
		
		\medskip
		
		If we now combine both claims we get that
		$$
		g(n,K_3) \leq n - \lfloor\frac{n}{3}\rfloor - 2  \leq  \frac{2n}{3} \, .
		$$
	\end{proof}
	
	\begin{figure}[h!]
		\centering
		\input{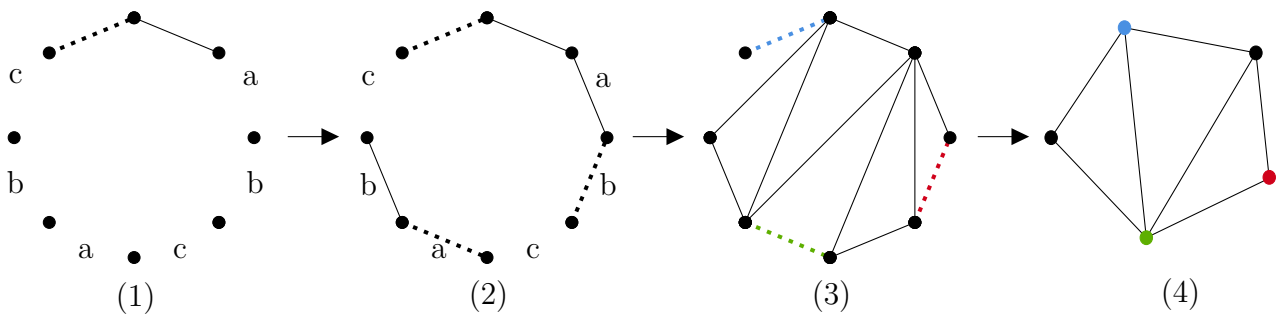}
		\caption{An example for Blocker's strategy for ECB. Constructor's edges are in solid lines and Blocker's edges are in dotted lines. (1) Labels of the outer edges after one move for each player; (2) Example of the graph later on; (3) Example of the graph at the end of the game; (4) Contraction of the Blocker edges.}
		\label{ECB : B strat}
	\end{figure}

	\section{Concluding remarks}
	
	In this paper, we studied the Constructor-Blocker game where Constructor wants to maximize the number of triangles in her graph while Blocker wants to minimize it. Figure~\ref{fig : table results} summarizes new and existing results, our contributions being in bold font.
	
	\begin{figure}[h!]
		\centering
		\begin{tabular}{|c|c|c|c|}
			\hline
			\rule[-10pt]{0pt}{25pt} $F$ & $ex(n,K_3,F)$ & $g(n,K_3,F)$ & ref. for $g(n,K_3,F)$\\
			\hline
			\noalign{\vskip 0pt}
			\hline
			\rule[-11pt]{0pt}{30pt} $\emptyset$ & $\frac{n^3}{6}$ & \boldmath$\frac{n^3}{48}$ & Thm \ref{thm : g(n,K_3,emptyset)}\\
			\hline
			\rule[-11pt]{0pt}{30pt} $P_5$ & $n$ & $\frac{n}{4}$ & \cite{patkos_constructor-blocker_2022}\\
			\hline
			\rule[-11pt]{0pt}{30pt} $P_6$ & $2n$ & \boldmath$\frac{n}{2}$ & Thm \ref{thm : g(n,K_3,P_6)}\\
			\hline
			\rule[-11pt]{0pt}{30pt} $S_4$ & $n$ & \boldmath $\frac{n}{3}$ & Thm \ref{thm : g(n,K_3,S_4)}\\
			\hline
			\rule[-11pt]{0pt}{30pt} $S_5$ & $2n$ & \boldmath$\frac{2n}{3} \leq \cdot \leq n$ & Thm \ref{thm : g(n,K_3,S_5)}\\
			\hline
			\rule[-11pt]{0pt}{30pt} $C_4$ & $\Theta(n^{3/2})$ & $ \Theta\left(n^{3/2}\text{e}^{-c\sqrt{\log(n)}}\right) \leq \cdot \leq O(n^{3/2})$ & \cite{balogh_constructor-blocker_2024}\\
			\hline
			\rule[-11pt]{0pt}{30pt} $F,\, \chi(F) = s>3$ & $\frac{(s-2)(s-3)}{6(s-1)^2}n^3$ & $\frac{(s-2)(s-3)}{48(s-1)^2}n^3$ & \cite{balogh_constructor-blocker_2024}\\ 
			\hline\noalign{\vskip 5pt}
			\hline
			\rule[-11pt]{0pt}{30pt} game & max. number of triangles & $g_{\text{game}}(n,K_3)$ & ref.\\
			\hline
			\noalign{\vskip 0pt}
			\hline
			\rule[-11pt]{0pt}{30pt} PCB & $3n$ & \boldmath$3n$ & Thm \ref{thm : PCB}\\
			\hline
			\rule[-11pt]{0pt}{30pt} ECB & $n$ & \boldmath$\frac{n}{2} \leq \cdot \leq \frac{2n}{3}$ & Thm \ref{thm : ECB}\\
			\hline
		\end{tabular}
		\caption{Recap table of asymptotic results}
		\label{fig : table results}
	\end{figure}
	
	Despite some similarities in the methods and strategies for the different forbidden subgraphs, we were not able to provide very generic statements that would handle many cases of $H$ and $F$. This would be really interesting and provide a deeper understanding of this problem. More precisely, we would be very interested in finding some strategies leading to the asymptotics of $g(n,K_3,S_k)$ and $g(n,K_3,P_k)$ for a general $k$. We would also like to close the gaps between the bounds for $g(n,K_3,S_5)$ and $g_\text{ECB}(n,K_3)$.
	
	Many other questions remain open. For instance, it is common in positional games to introduce a bias. Instead of the two players each claiming one edge during their turns, we can allow Constructor to claim $a$ edges and Blocker to play $b$ edges. This bias could change the score of the game, even if $a=b$ a priori. It would be interesting to see how the score of the game evolves with $a$ and $b$.
	
	Finally, even without improving the bounds, we would like to have nicer strategies, especially for Blocker. In most of the proofs, we describe the strategies for Blocker very locally, with a case study, but we think there could be easier strategies for Blocker giving the same results.

	\bibliographystyle{amsplain}
	\bibliography{references}

\end{document}